\documentclass[12pt]{amsart}
\usepackage{amscd,amssymb, amsmath, url,graphicx,color, hyperref}
\usepackage[all]{xy}

\newtheorem{thm}{Theorem}
\newtheorem{cor}[thm]{Corollary}
\newtheorem{lemma}[thm]{Lemma}
\newtheorem{prop}[thm]{Proposition}
\theoremstyle{definition}

\newcommand{\sections}{\renewcommand{\thethm}{\thesection.\arabic{thm}}
           \setcounter{thm}{0}}
\newcommand{\nosubsections}{\renewcommand{\thethm}{\thesection.\arabic{thm}}
           \setcounter{thm}{0}}

\newcommand{\linnum}{\stepcounter{thm}\tag{\thethm}}

\newcommand{\co}{\colon}

\newcommand{\bbC}{\mathbb{C}}

\newcommand{\bbH}{\mathbb{H}}
\newcommand{\bbQ}{\mathbb{Q}}
\newcommand{\bbR}{\mathbb{R}}
\newcommand{\bbZ}{\mathbb{Z}}

\newcommand{\cH}{\mathcal{H}}

\newcommand{\cT}{\mathcal{T}}
\newcommand{\cZ}{\mathcal{Z}}
\newcommand{\za}{\alpha}
\newcommand{\zb}{\beta}
\newcommand{\zd}{\delta}

\newcommand{\zg}{\gamma}
\newcommand{\zh}{\eta}
\newcommand{\zi}{\iota}
\newcommand{\zj}{\psi}
\newcommand{\zl}{\lambda}
\newcommand{\zm}{\mu}
\newcommand{\zn}{\nu}

\newcommand{\zr}{\rho}

\newcommand{\zs}{\sigma}
\newcommand{\zt}{\tau}
\newcommand{\zv}{\varphi}

\newcommand{\zF}{\Phi}
\newcommand{\zG}{\Gamma}
\newcommand{\zJ}{\Psi}
\newcommand{\zL}{\Lambda}

\newcommand{\bD}{\mathbf{D}}
\newcommand{\Z}{\mathbb{Z}}

\newcommand{\geomsize}{\mathrm{geomsize}}
\newcommand{\algsize}{\mathrm{algsize}}
\newcommand{\height}{\mathrm{ht}}
\newcommand{\nonslope}{\odot}

\newcommand{\diam}{\mathrm{diam}}

\newcommand{\Mod}{\mathrm{Mod}}
\newcommand{\cW}{\mathcal{W}}
\newcommand{\cM}{\mathcal{M}}
\newcommand{\cE}{\mathcal{E}}
\newcommand{\norm}[1]{\left\lVert#1\right\rVert}

\begin{document}
\title[Rationality is decidable]{Rationality is decidable for Nearly Euclidean Thurston maps}

\begin{author}{William Floyd}
\address{Department of Mathematics\\ Virginia Tech\\
Blacksburg, VA 24061\\ USA}
\email{floyd@math.vt.edu}
\urladdr{http://www.math.vt.edu/people/floyd}
\end{author}

\begin{author}{Walter Parry}
\address{Department of Mathematics and Statistics\\ 
Eastern Michigan University\\
Ypsilanti, MI 48197\\ USA}
\email{walter.parry@emich.edu}
\end{author}

\begin{author}{Kevin M. Pilgrim}
\address{Department of Mathematics, Indiana University, Bloomington, 
IN 47405, USA}
\email{pilgrim@indiana.edu}
\end{author}

\date{\today}

\begin{abstract} Nearly Euclidean Thurston (NET) maps are described by
simple diagrams which admit a natural notion of size.  Given a size bound $C$, there are finitely many diagrams of size at most $C$.    Given a NET map $F$ presented by a diagram of size at most $C$, the problem of determining whether $F$ is equivalent to a rational function is, in theory, a finite computation.  We give bounds for the size of this computation in terms of $C$ and one other natural geometric quantity. This result partially explains the observed effectiveness of the computer program {\tt NETmap} in deciding rationality. 
\end{abstract}



\thanks{}
\maketitle

\tableofcontents 

\sections

\section{Introduction}\nosubsections

A \emph{Thurston map} $F: S^2 \to S^2$ is an orientation-preserving
branched covering of degree at least two for which the
\emph{postcritical set} $P_F=\cup_{n>0}F^n(C_F)$ is finite; here $C_F$
denotes the set of branch points of $F$.  Thurston maps originally
arose in fundamental classification problems in
one-complex-dimensional dynamics \cite{DH}.  If a finite set $P
\subset S^2$ is given, the collection of isotopy classes relative to $P$ of Thurston
maps $F$ with $F(P)\subset P$ and $P_F\subset P$ forms a countable semigroup under composition.  The
collection  $\mathcal{H}$ of all isotopy classes of Thurston maps obtained by
pre- and post-composing a given Thurston map $F$ with
orientation-preserving homeomorphisms fixing $P$--called its
\emph{pure modular group Hurwitz class}--has a very rich
algebraic structure; see \cite{BD0}.  Through this more recent
algebraic perspective, Thurston maps may now be fruitfully regarded as
analogs of elements of the well-studied mapping class groups. For
example, a finite collection of pairwise disjoint, mutually
non-homotopic curves on a surface that is left invariant by a mapping
class--a so-called \emph{reducing multicurve}--is typically regarded
as an obstruction to \emph{geometrization} of this class. Moreover, among such reducing multicurves, there exists a canonical one \cite[Cor. 13.3]{fm}.  Here,
geometrization means finding a representative of the class with
optimal geometric properties--a periodic or pseudo-Anosov map. On the
Thurston map side, W. Thurston's fundamental characterization and
rigidity theorem says the following. Typically, in the absence of an
obstruction--again defined as a multicurve with certain invariance
properties--a Thurston map is conjugate-up-to-isotopy, or
\emph{equivalent}, to a rational function, unique up to M\"obius
transformations. Again, among all such obstructions, there exists a canonical one \cite{kmp:canonical}, \cite{S}.

The current state of algorithms for computing with mapping class groups is quite advanced. For example, Margalit, Strenner, and Yurtta\c{s} \cite{MSY} have announced  a quadratic-time algorithm for locating reducing curves (and more). That is, given a concrete presentation of a mapping class in terms of a standard generating set, their algorithm (and implementation, {\tt Macaw}, available at \url{https://github.com/b5strbal/macaw}) locates in quadratic time a certain canonically defined reducing curve, if it exists. 

Our main result here is an incremental, but first, such quantitative
result for Thurston maps.  It is incremental because (i) it is
restricted to a special class of Thurston maps, called nearly
Euclidean Thurston (NET) maps, that we have been studying in
\cite{cfpp, fkklpps, fpp1, fpp2, P}, and (ii) our bounds are
unfortunately not quite explicit, due to our inability to effectively
estimate a certain geometric constant; see Section
\ref{secn:pf_thm_short_cf} below.  Our methods, however, employ
geometric themes also encountered in the study of mapping class
groups. In particular, continued fractions play an essential role in our study. 

A Thurston map $F$ is a NET map if $\#P_F=4$ and every branch point is simple, i.e has local degree $2$. Among NET maps are the atypical \emph{Euclidean} Thurston maps, which are equivalent to quotients of affine planar maps, and the question of their rationality is equivalent to an easy-to-verify condition on the eigenvalues of the derivative of this affine endomorphism.  Our sole focus here is on the more typical, non-Euclidean NET maps. A NET map is non-Euclidean, or typical, if and only if some element of $P_F$ is critical. 

Each NET map, up to equivalence, is describable by a simple \emph{NET map presentation diagram} $\bD$; see \cite{fpp1}.  Figure \ref{fig:rabbitpren} shows a presentation diagram for the Douady rabbit polynomial. 
  \begin{figure}
\centerline{\includegraphics{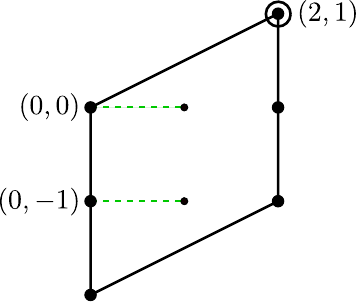}}
\caption{A presentation diagram for the Douady rabbit polynomial $f(z)=z^2+c$, where $0 \mapsto c \mapsto c^2+c \mapsto 0$, $\Im(c)>0$. }
\label{fig:rabbitpren}
  \end{figure}
Each presentation diagram shows a parallelogram spanned by two
linearly independent integer lattice vectors $2\lambda_1,
\lambda_2$. In Figure \ref{fig:rabbitpren}, $\lambda_1=(0,-1)$ and
$\lambda_2=(2,1)$.  The \emph{geometric size} $\geomsize(\bD)$ of a
NET map presentation diagram $\bD$ is defined as $\max\{|\lambda_1|,
|\lambda_2|\}$.  A NET map presentation diagram $\bD$ determines a
unique isotopy class of NET maps $F=F_{\bD}$. Here is how; see
\cite[Lemma 3.1(3)]{fpp1} for details. Let $\pi: \bbR^2 \to \bbR^2/\Gamma=:S^2$
denote the quotient map to the standard square pillowcase, where
$\Gamma = \langle x \mapsto 2\lambda - x:\lambda \in \bbZ^2\rangle$.
The set $P\subset S^2$ is the image of $\bbZ^2$ under the natural
projection. Let $b$ be the integer vector whose coordinates are
circled in the diagram. Let $\Psi: \bbR^2 \to \bbR^2$ be the affine
map $\Psi(x)=Ax+b$ where $A=[\lambda_1, \lambda_2]$ is the
corresponding column matrix. Let $G: (S^2,P) \to (S^2,P)$ be the
induced affine endomorphism; it is a Euclidean NET map. Finally, let
$F=G\circ H$, where $H: S^2 \to S^2$ is a homeomorphism obtained by
pushing the starting points along the images of
$\pi(\Psi^{-1}(\alpha_j))$ in the direction indicated by the dashed
green arcs $\za_j$ in Figure 1. The resulting map $F$ is a Thurston
map with $P_F\subseteq P$.  (We remark that in the notation of \cite[Lemma 3.1(3)]{fpp1}, the map $F=G\circ H$ where $G=\overline{id}\circ \phi$ and $H=\phi^{-1}\circ h \circ \phi$.) The map $F$ is a typical non-Euclidean NET map if
and only if there is at least one nontrivial dashed green arc and
$\#P_F=4$, so that $P_F=P$; these conditions are easy to check. The
degree of $F$ is the determinant $\det(A)$.

In this work, we always assume by conjugation that all NET maps are defined on the standard square pillowcase $S^2$ and have postcritical set $P$ as defined in the previous paragraph.

On $S^2-P$, each isotopy class of simple closed essential
nonperipheral unoriented curves contains a representative which lifts
to a Euclidean geodesic of rational slope. In this way we obtain a
bijection between such classes of curves and the set of their slopes,
the set of extended rational numbers, $\overline{\bbQ}=\bbQ \cup
\{\frac{1}{0}\}$. Given a NET map $F$ and a simple closed curve $\zg$
in $S^2-P$ with slope $s$, we let $c(s)$ denote the number of
essential nonperipheral connected components of $F^{-1}(\zg)$ and we
let $d(s)$ denote the (necessarily common) degree with which these
connected components map to $\zg$ \cite[Section 5]{cfpp}.  The
\emph{multiplier} of $s$ is $\zd(s):=c(s)/d(s)$.  We let $\nonslope$
stand for the union of isotopy classes of peripheral and inessential
curves. Via pullback, a NET map $F$ induces a \emph{slope function}
$\mu_F: \overline{\bbQ} \to \overline{\bbQ}\cup \{\nonslope\}$. It
also induces similarly an analytic self-map $\sigma_F: \bbH \to \bbH$,
where the upper half-plane $\bbH$ is naturally identified, via the
Weierstrass theory, with the Teichm\"uller space of $(S^2, P)$
\cite{DH}. Selinger \cite{S} shows that $\sigma_F$ extends canonically
to the Weil-Petersson completion of $\bbH$, which adjoins the set
$\overline{\bbQ}$ of cusps.  If $s,s'\in \overline{\bbQ}$ are slopes,
then $\mu_F(s)=s' \iff \sigma_F(t)=t'$, where $t=-1/s$ and $t'=-1/s'$
are the corresponding cusps.  We define the \emph{height} of
$\frac{p}{q}\in \overline{\bbQ}$ for integers $p$, $q$ with
$\gcd(p,q)=1$ as $\height(p/q)=\max\{|p|, |q|\}$.

An obstructed non-Euclidean NET map $F$ necessarily has a unique obstruction,
i.e. there is a single extended rational obstruction slope.  It is a
fixed point of $\zm_F$.  So we are interested in fixed points of
$\zm_F$.  Since fixed points of $\zm_F$ are negative reciprocals of
fixed points of $\zs_F$, we are also interested in cusps fixed by
$\zs_F$.  Most of our work involves making estimates about $\sigma_F$.

An elementary argument shows there are at most $O(C^{12})$
presentation diagrams $\bD$ with $\geomsize(\bD)\leq C$. Thus there is
an upper bound $H(C)$ for the height of the slope of an obstruction of
an obstructed NET map $F=F_{\bD}$ with $\geomsize(\bD) \leq C$. Less
obviously, there is usually also an upper bound $H(C)$ for the height
of every fixed point of $\zm_F$.  Our main result gives an estimate
for $H(C)$.

\begin{thm}[Height bound for fixed cusps]
\label{thm:finite} 
Let $\bD$ be a non-Euclidean NET map presentation diagram with
$\geomsize(\mathbf{D})\leq C$, and $F:=F_{\mathbf{D}}$ the
corresponding NET map.  Then there exists a positive integer $H(C)$
with the following properties.
\begin{enumerate} 
  \item Let $t$ be a cusp fixed by $\zs_F$, and assume that $t$ is the
negative reciprocal of the unique obstruction slope if $F$ has an obstruction
with multiplier 1.  Then 
\[ \height(t) \leq H(C)\le (AC)^{9\cdot 19^N}.\] 
The exponent $N$ is an upper bound on the length of the continued
fraction expansion of a cusp fixed by $\zs_F$, and satisfies
$N\lesssim (1-c(\cH))^{-1}\cdot \log C$ where $0 < c(\cH)<1$.  The
implied constant and the real number $A$ are universal.  The constant
$c(\cH)$ depends only on the modular group Hurwitz class $\cH$ of $F$
and is defined as an upper bound on the size of the Teichm\"uller
(hyperbolic) derivative of $\sigma_F$ on a universal cocompact subset
of Teichm\"uller space, independent of $F$.
  \item The function $H(C)$ cannot be taken to be a polynomial in $C$
even if we restrict attention to only negative reciprocals of slopes
of obstructions.
  \end{enumerate}
\end{thm}

By testing for obstructions all curves of slopes up to height $H(C)$, we obtain a very special case of the more general but non-quantitative result of Bonnot, Braverman, and Yampolsky \cite{BBY}; see also \cite{BD5}: 

\begin{cor}[Rationality is decidable for NET maps]
\label{cor:decidable}
Suppose $F=F_{\bf D}$. The question, ``Is $F$ equivalent to a rational map?'' is decidable.
\end{cor}

Before discussing the proof, we make some remarks about the last, sharpness assertion (2) in Theorem \ref{thm:finite}. 

\begin{enumerate}
\item There is a competing natural notion of size for a NET map presentation diagram. Following Bartholdi-Nekrashevych \cite{BN}, one can rather naturally associate, using the geometry of the presentation diagram, a wreath recursion on the fundamental group $G$ of $S^2-P$ induced by $F$.  As basepoint for $G$, one may take e.g. $\pi(1/2, 1/2)$.  Since $S^2$ is presented as $\bbR^2/\Gamma$, there is a natural generating set $\{g_1, \ldots, g_4\}$ for $G$ satisfying $g_1g_2g_3g_4=1$: the $g_i$'s run from the basepoint to a corner of the front square of the pillowcase, and then loop around the corner point. By definition, the wreath recursion is a certain homomorphism $\Phi: G \to G^d \rtimes S_d$, $d=\deg(F)$, given by $g \mapsto (g|_1, \ldots, g|_d)\sigma(g)$ where $g|_j \in G$.  The definition of $\Phi$ depends on a choice of a collection of homotopy classes of arcs joining the basepoint to each of its preimages.  The diagram gives a natural way to do this. Define  the \emph{algebraic size} $\mathrm{algsize}(\bD)$ of $\bD$ as $\max\{\norm{g_i|_j} : j \in \{1, \ldots, d\}, i=1,\dotsc,4\}$.  It is not difficult to show that $\algsize(\bD) \leq 42 \cdot \geomsize(\bD)$. However, there is no similar bound in the other direction: if $\lambda_1=(n,0), \lambda_2=(0,n)$, $n$ is very large, and there is just one nontrivial green arc $\alpha_1$ joining $(0,0)$ to e.g. $(1,1)$, then the algebraic size will be small while the geometric size will be large.

\item Two different presentation diagrams may yield isotopic--not just
equivalent--Thurston maps. The natural action of $\mathrm{SL}_2(\bbZ)$
on presentation diagrams--by just applying the matrix to the whole
diagram--corresponds to post-composition with the induced map on the
quotient pillowcase. On the level of isotopy classes, however, this
action is not free: for any typical Thurston map, there are always
nontrivial mapping classes containing representatives which lift under
$F$ to homeomorphisms isotopic to the identity \cite[Theorem
6.3]{kps}.  
%Thus a perhaps more natural notion of size would
%associate, to a NET map $F_1: (S^2,P) \to (S^2, P)$ on the standard
%square pillowcase, the minimum geometric size of a diagram $\bD$
%for which $F_1$ is isotopic to $F_{\bD}$ relative to $P$. However, we
%do not know how to compute this minimum size quantity.

\item The previous remark implies that a given NET map $F_1$ is
isotopic relative to $P$ to infinitely many maps $F_{\bD_k}$, so that
$\geomsize(\bD_k) \to \infty$ as $k\to\infty$.
Section~\ref{sec:largeratio} below, however, shows the sharpness
statement (2) of Theorem \ref{thm:finite} by exhibiting an infinite
sequence of degree 2 Thurston maps $F_n=F_{\bD_n}$ with $n\le
\geomsize(\bD_n)\le n+2$ and obstruction slope $s_n$ satisfying
$\height(s_n) \sim n^{\log_2(n)/2}$.  In particular, $\height(s_n) \to
\infty$.  Since the obstructions are unique and their slope heights
tend to infinity, the maps $F_n$ define an infinite set of distinct
homotopy classes.  Nonetheless, we prove in
Theorem~\ref{thm:conjugation2} that these maps lie in only three
equivalence classes.
\end{enumerate}

There are two main ingredients in the proof of Theorem
\ref{thm:finite}.  The first ingredient, Theorem \ref{thm:short_cf},
bounds the length of the continued fraction expansion of a cusp fixed
by $\zs_F$ in terms of the size bound $C$ and an estimate $c(\cH)<1$
for the hyperbolic derivative $||d\sigma_F||$ on a certain universal
cocompact subset of $\bbH$. We are unfortunately unable to bound
$c(\cH)$ explicitly away from $1$ in terms of $C$.

The second ingredient, Theorem \ref{thm:qxi}, exploits a key feature
of NET maps: given a slope $s$, there is an algorithm for calculating
the image (under pullback) slope $s'=\mu_F(s)$, as well as the number
of essential preimage components $c(s)$ of a curve with slope $s$ and
the (necessarily common) degree $d(s)$ by which they map under $F$
\cite[Section 5]{cfpp}.  A slope $s$ fixed by $\zm_F$ is the slope of
an obstruction if and only if the \emph{multiplier},
$\delta(s):=c(s)/d(s)$, satisfies $\delta(s) \geq 1$.  In terms of
cusps, a cusp $t$ corresponds to an obstruction $s=-1/t$ if and only if 
$\sigma_F(B)\subset B$ for each horoball $B$ tangent to $t$.
Similarly, a cusp is fixed if and only if for each horoball $B$
tangent to $t$, the image $\sigma_F(B)$ is contained in another
horoball $B'$ tangent to $t$ which cannot be much larger than $B$. By
exploiting this observation, one can show that each datum $(s, s',
c(s), d(s))$ gives rise to an \emph{excluded interval} in which fixed
cusps cannot lie \cite{P}.  Theorem \ref{thm:qxi} extends the
Half-Space Theorem of \cite[Theorem 6.7]{cfpp} in two ways: it
excludes fixed cusps, not just negative reciprocals of slopes of
obstructions, and it applies in wider generality to certain
exceptional cases not treated there. Related excluded intervals are
used in the computer program {\tt NETmap} in its attempt to answer the
question of whether a given map $F=F_{\bD}$ is equivalent to a
rational map. Under robust observed conditions, the program succeeds:
on every example in the database of 40,000+ examples tabulated at the
website \cite{NET}, the rationality question is answered, modulo our
faith in numerical precision.  The height bound in Theorem
\ref{thm:finite} partially explains this observed
effectiveness. 

Here are the precise statements of these two ingredients.  In each, we
denote by $F=F_{\mathbf{D}}$ where $\geomsize(\mathbf{D})\leq C$ and $\mathbf{D}$ is non-Euclidean.  We
let $[a_0,\dotsc,a_n]$ denote the regular continued fraction with
partial quotients $a_0,\dotsc,a_n$.

\begin{thm}[Fixed cusps have short continued fraction expansions]
\label{thm:short_cf} There exists a positive integer $N=N(C,
\mathcal{H})$ such that if $t=[a_0,\dotsc,a_n]$ is a cusp for which
$\sigma_F(t)=t$, then 
$$n \leq N\lesssim (1-c(\cH))^{-1}\cdot \log C$$
and $0 < c(\cH)<1$. The constant $c(\cH)$ depends only on the modular
group Hurwitz class $\cH$ of $F$ and is an upper bound on the size of
the Teichm\"uller (hyperbolic) derivative of $\sigma_F$ on a universal
cocompact subset of Teichm\"uller space, independent of $F$. The
implicit constant is universal.
\end{thm}

\begin{thm}[Quantitative excuded intervals] \label{thm:qxi} There
exist positive numbers $R=R(C)$ and $\rho=\rho(C,h)$ for every
positive integer $h$ with the following property.  Let $t$ be a cusp
with $\height(t)\le h$. Assume that if $\zs_F(t)=t$ then the
multiplier $\delta(-1/t)\neq 1$.  Then
\begin{enumerate}
\item if $t=1/0$, the deleted-at-infinity interval $I_t:=(-\infty, -R)
\cup (R,+\infty)$ contains no cusps fixed by $\zs_F$;
\item if $t \neq 1/0$, the deleted interval $I_t:=(t-\rho, t) \cup (t,
t+\rho)$ contains no cusps fixed by $\zs_F$.
\end{enumerate}
Furthermore $R\lesssim C^8$ and $\rho\gtrsim C^{-8}h^{-18}$. 
\end{thm}

The assumption that the multiplier of a fixed cusp $t$ satisfies
$\delta(-1/t)\neq 1$ is necessary.  Here is an example.
Figure~\ref{fig:omit} shows a presentation diagram for a NET map $F$.
The map $F$ has an obstruction $\gamma$ with slope $\infty $ and
multiplier 1. The curve of slope $0$ is fixed.  Let $T$ be a Dehn
twist about $\gamma$. It is easy to check directly that $T\circ F$ is
isotopic to $F \circ T$ relative to $P$. Since we may choose $T$ so
that $\sigma_T(z)=\frac{z}{2z+1}$, it follows that $\sigma_F(1/n)=1/n$
for each even integer $n$.  Thus every deleted interval about the
fixed cusp $0$ contains infinitely many fixed cusps.\\

\noindent {\bf Organization.} Section \ref{thm:pf_finite} derives the height bound of Theorem \ref{thm:finite} from Theorems \ref{thm:short_cf} and \ref{thm:qxi}.  Section \ref{secn:surjective} contains a result of independent interest, Theorem \ref{thm:trivial_or_surjective}, which asserts that for Thurston maps with $\#P_F=4$, $\mu_F$ is surjective if and only if it is nontrivial. Section \ref{secn:size} contains several results showing how the size bound $C$ controls the geometry of $F$, of $\mu_F$,  and of $\sigma_F$.  Sections \ref{secn:pf_thm_short_cf} and \ref{secn:pf_thm_qxi} give, respectively, the proofs of Theorems \ref{thm:short_cf} and \ref{thm:qxi}. Section \ref{sec:largeratio} concludes the proof of Theorem \ref{thm:finite} by establishing the sharpness assertion: it gives a sequence of obstructed examples for which the heights of the obstructions grow more than polynomially in the geometric sizes.  \\

\noindent{\bf Acknowledgements} K. Pilgrim was supported by Simons grant 245269. The authors also acknowledge support from the AIM ``SQuaRE'' program.

  \begin{figure}
\centerline{\includegraphics{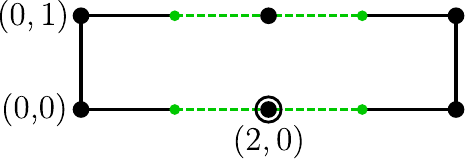}}
\caption{A presentation diagram of a NET map which shows that the
assumption that $\zd(-1/t)\ne 1$ is necessary }
\label{fig:omit}
  \end{figure}

 \section{Notation}
 \label{secn:notation}
 
Before proceeding to the proofs, we establish  in this section some notation, conventions, and terminology.  

All surface branched coverings and homeomorphisms considered here are orientation-preserving. 

{\bf The square pillowcase group and its subgroups.}  Our two-sphere $S^2$ and set $P$ of four points are as defined in the
introduction.  Its affine automorphism group
$\text{PSL}(2,\bbZ)\ltimes (\bbZ/2\bbZ)^2$ is then naturally
identified with the (impure) mapping class group
$\text{Mod}(S^2,P)$. We refer to the elements of $(\bbZ/2\bbZ)^2$ as
translations.  The pure mapping class group $\text{PMod}(S^2,P)$ is identified with  $\overline{\Gamma}(2)=\{A\in \mathrm{SL}_2(\bbZ): A
\equiv I \bmod 2\}/\{\pm I\}<\text{PSL}_2(\bbZ)$.   

{\bf Slopes and cusps.}  Slopes are denoted using the letter $s$; cusps denoted using the
letter $t$. Both slopes and cusps are extended rational numbers,
elements of $\overline{\bbQ}=\bbQ\cup \{\infty \}$.  

{\bf Thurston maps.}  All Thurston maps considered here will be defined as
maps $F: (S^2,P) \to (S^2,P)$ on the standard square pillowcase and by definition are orientation-preserving. The degree of $F$ is denoted $\deg(F)$. Under pullback induced by $F$, images of slopes
and cusps under $\mu_F$ and $\sigma_F$ are generally denoted with
primes.

{\bf The upper half-plane $\bbH$.}  To simplify notation for the present discussion, we denote by $G=\text{PSL}(2,\bbZ)\ltimes (\bbZ/2\bbZ)^2 \cong \text{Mod}(S^2, P)$.  

%We denote by $\overline{\Gamma}(2)=\{A\in \mathrm{SL}_2(\bbZ): A
%\equiv I \bmod 2\}/\{\pm I\} < \mathrm{PSL}_2(\bbZ)$ and identify this
%with the pure mapping class group $\text{PMod}(S^2,P)$ of $(S^2, P)$.
%Let $G=\text{PMod}(S^2,P)$.  

Just as $\zs_F:\bbH\to \bbH$ is gotten
from $F$ by pulling back complex structures, we obtain a pullback map
$\zs_h:\bbH\to \bbH$ for every $h\in G$.  Note that translations act trivially.   The action is not quite the obvious one, since it is induced via pullback:  as in \cite[Section
6]{cfpp} and \cite[Proposition 4.1]{fpp2}, if $h$ has matrix
$\left[\begin{smallmatrix}a & b \\ c & d \end{smallmatrix}\right]\in
\mathrm{SL}_2(\bbZ)$, then $\zs_h$ has matrix
$\left[\begin{smallmatrix}d & b \\ c & a
\end{smallmatrix}\right]$.  In this way we obtain a right action of $G$ on
$\bbH\cup \overline{\bbQ}$ so that $h.z=\zs_h(z)$ for $h\in G$ and
$z\in \bbH\cup \overline{\bbQ}$.

We denote by $d_{\bbH}$ the hyperbolic distance on $\bbH$ induced by
the line element $ds=\frac{|dz|}{\Im(z)}$.  

{\bf Liftables.}  Given a Thurston map $F:
(S^2,P) \to (S^2,P)$, there is a distinguished subgroup $G_F$ of $\text{PMod}(S^2,P)$--the ``pure liftable'' elements--comprised of those classes
represented by homeomorphisms $h$ for which $h\circ F$ is isotopic
relative to $P$ to $F \circ \tilde{h}$ for some $\tilde{h}\in \mathrm{PMod}(S^2,P)$. The
assignment $h \mapsto \tilde{h}$ defines the \emph{virtual
endomorphism} $\phi_F: G_F \to \text{PMod}(S^2,P)$.  The pullback map $\sigma_F$ and
virtual endomorphism $\phi_F$ are related by the functional equation
$\sigma_F(h.\tau)=\phi_F(h).\sigma_F(\tau)$. See \cite{kps}.

\section{Deriving the height bound}
\label{thm:pf_finite}\nosubsections

In this section, we deduce statement 1 of Theorem \ref{thm:finite}
from Theorems \ref{thm:short_cf} and \ref{thm:qxi}.  Statement 2 is
proved in Section~\ref{sec:largeratio}.

\begin{proof} We will use the following lemma, which contains three
standard facts about continued fractions.

\begin{lemma}\label{lemma:cfracs} Let $[a_0,\dotsc,a_n]$ be a
regular continued fraction.  So $a_0,\dotsc,a_n$ are integers, and
$a_k>0$ if $k\ge 1$.  Let $\frac{p_k}{q_k}=[a_0,\dotsc,a_k]$
for $0\le k\le n$ be its convergents.  Set $p_{-1}=1$, $q_{-1}=0$,
$p_{-2}=0$ and $q_{-2}=1$. Then the following statements hold.
\begin{enumerate}
  \item $p_k=a_kp_{k-1}+p_{k-2}$ and $q_k=a_kq_{k-1}+q_{k-2}$ for
$0\le k\le n$
  \item
$\left|\frac{p_{k-1}}{q_{k-1}}-\frac{p_k}{q_k}\right|=\frac{1}{q_{k-1}q_k}$
for $k>0$
  \item $\frac{p_0}{q_0}<\frac{p_2}{q_2}<\frac{p_4}{q_4}<\cdots 
<\frac{p_5}{q_5}<\frac{p_3}{q_3}<\frac{p_1}{q_1}$
\end{enumerate}
\end{lemma}

\noindent{\bf Continued fractions.}. We will work with finite regular continued fractions throughout this
proof.  When we write $[a_0,\dotsc,a_n]$, it is assumed that this
is a regular continued fraction with integers $a_0,\dotsc,a_n$ such
that $a_i>0$ for $i\in\{1,\ldots,n\}$.

Every rational number $t$ can be represented in exactly two ways as a
finite regular continued fraction:
  \begin{equation*}
t=[a_0,\dotsc,a_{n-1},a_n,1]=[a_0,\dotsc,a_{n-1},a_n+1].
  \end{equation*}
We will work with both of these representations of $t$.  So despite
writing equations as in the last display (an abuse of notation), we
must distinguish between finite regular continued fractions and the
rational numbers which they represent.  The last display is to be
interpreted as saying that both continued fractions represent $t$.  
Given one of the two continued fraction expansions of a rational number $t$, the sequence of convergents $p_n/q_n$ associated to this expansion is well-defined; it is not, however, well-defined given merely $t$ itself.   

Finally, we will find
it convenient to extend the definition of convergent as follows. Given a rational number $t$ and an expansion  $t=[a_0,\dotsc,a_n]$ and an integer $k>n$, we define the
$k$-convergent of this expansion to be the rational number $t$.

We begin the proof of Theorem~\ref{thm:finite} proper.  Suppose that
$F=F_{\mathbf{D}}$ and $\geomsize(\mathbf{D})\leq C$.  Let $N$ be the
bound on the length of the continued fraction expansion of a cusp $t$
satisfying $\sigma_F(t)=t$ given by Theorem \ref{thm:short_cf}. We
will show the existence of a finite set of cusps $\cT_N$ depending on
only $C$ and $N$ such that the cusp $t$ in the statement of
Theorem~\ref{thm:finite} must lie in $\cT_N$.  We will construct
$\cT_N$ as an ascending union of finite sets $\cT_{-1} \subseteq \cT_0
\subseteq \cT_1 \subseteq \cdots\subseteq \cT_N$. The sets $\cT_k$
will be defined inductively as a collection of cusps whose continued
fraction expansions satisfy a certain property.

To initialize, we set $\cT_{-1}:=\{\frac{1}{0}\}$. 

The construction of $\cT_0$ is slightly different from the
construction of $\cT_k$ in general, so we perform it separately. If
$\frac{1}{0}$ is the negative reciprocal of the slope of an
obstruction with multiplier 1, then there is nothing to prove.
Otherwise let $I_{\frac{1}{0}}$ be the excluded deleted-at-infinity
interval provided by Theorem \ref{thm:qxi}, and let $a$ be the largest
negative integer in $I_{\frac{1}{0}}$.  We let
  \begin{equation*}
\cT_0=\cT_{-1} \cup 
\{a_0\in \bbZ: a_0\notin I_{\frac{1}{0}}\}\cup \{a\}.
  \end{equation*}
The point of choosing $a$ in this way is that if $b_0$ is a negative
integer in $I_\frac{1}{0}$ other than $a$, then $b=[b_0,\dotsc,b_m]\in
I_\frac{1}{0}$ for every choice of positive integers $m$ and
$b_1,\dotsc,b_m$ because $b\le a$.  It is clear that if $b_0$ is a
positive integer in $I_\frac{1}{0}$, then $b=[b_0,\dotsc,b_m]\in
I_\frac{1}{0}$ for every choice of positive integers $m$ and
$b_1,\dotsc,b_m$ because $b>b_0$.  So if $[a_0,\dotsc,a_n]$ represents
a cusp fixed by $\zs_F$, then $[a_0]\in \cT_0$.

We will want to control heights of elements of $\cT_k$ too. We will also do this inductively. For the base case, note from Theorem \ref{thm:qxi} that there exists a universal constant $A_0$ for which $\max\{\height(t): t \in \cT_0\}\leq H_0:= A_0\cdot C^8$. 

Here is the inductive step.  Suppose $k \geq 1$.  Our inductive
hypothesis is (i) that the set $\cT_{k-1}$ contains the
$(k-1)$-convergent of every finite regular continued fraction which
represents a cusp fixed by $\zs_F$ and (ii) we have constructed an
upper bound $H_{k-1}<\infty$ on the heights of the elements of
$\cT_{k-1}$.  Note that the height bound implies $\cT_{k-1}$ is
finite.  If $\cT_{k-1}$ contains the negative reciprocal of the slope
of an obstruction with multiplier 1, then the construction stops here.
We set $\cT_N=\cT_{k-1}$ and proceed to the estimation of $H_N$ at
the end of this proof.  Otherwise, we may assume that every cusp $u\in
\cT_{k-1}$ has a deleted interval $I_u$ as in Theorem~\ref{thm:qxi}.

Here is the definition of $\cT_k$: \[ \cT_k:=\cT_{k-1}\cup
\{t=[a_0,\dotsc,a_k] :u=[a_0,\dotsc,a_{k-1}] \in\cT_{k-1}, t \not\in
I_u\}.\] In other words: for each element $u=[a_0,\dotsc,a_{k-1}]$ of
$\cT_{k-1}$, consider those cusps $t$ represented by a continued
fraction which extends that of $u$ by one more partial quotient \[
t=[a_0,\dotsc, a_{k-1}, a_k].\] Then $t$ lies in $\cT_k$ if and only
if $t \not\in I_u$.

We now check the inductive step. Suppose $t=[a_0,\dotsc,a_n]$ is a
cusp fixed by $\zs_F$. By the inductive hypothesis, $\cT_{k-1}$
contains the $(k-1)$-convergent of $[a_0,\dotsc,a_n]$.  We first claim
that $\cT_k$ contains the $k$-convergent of $[a_0,\dotsc,a_n]$.  To
verify this, there are a few cases to consider, depending on how $n$
compares with $k$:
\begin{itemize} 
  \item If $n \leq k-1$, then the $k-1$ and $k$-convergent of
$[a_0,\dotsc,a_n]$ is $t$, and so $t \in \cT_{k-1}\subset \cT_k$.
  \item If $n=k$, then the $k$-convergent of $[a_0,\dotsc,a_n]$ is
$t$.  Let $u=[a_0,\dotsc,a_{k-1}]$.  The fixed cusp $t$ cannot lie in
the excluded interval $I_u$, so $t\in \cT_k$ by the definition of
$\cT_k$.
  \item Now suppose $n>k$.  Let $u=[a_0,\dotsc,a_{k-1}]$ and
$v=[a_0,\dotsc,a_k]$.  We proceed by contradiction: suppose $v\not\in
\cT_k$.  Recall the inductive hypothesis asserts $u\in \cT_{k-1}$. The
definition of $\cT_k$ then implies that $v\in I_u$. The rationals $u$,
$v$ are consecutive convergents of $[a_0,\dotsc,a_n]$.  So line 3 of
Lemma~\ref{lemma:cfracs} implies that $t$ is strictly between $u$ and
$v$. Since $v\in I_u$, a deleted excluded interval about $u$, we
conclude $t\in I_u$.  But $t$ is a fixed cusp, so this is impossible
by Theorem \ref{thm:qxi}.  We conclude $v \in \cT_k$.
\end{itemize}

We next estimate the heights of the elements of $\cT_k$ in terms of
$H_{k-1}$.  Let $t=[a_0,\dotsc,a_k]\in \cT_k$.  Let $a=a_k$, let
$\frac{p}{q}=[a_0,\dotsc,a_{k-1}]$ and let
$\frac{p'}{q'}=[a_0,\dotsc,a_{k-2}]$ if $k\ge 2$ and let
$\frac{p'}{q'}=\frac{1}{0}$ if $k=1$.  Of course, $p$, $q$, $p'$, $q'$
are integers with $q>0$, $q'\ge 0$ and $\gcd(p,q)=\gcd(p',q')=1$.
Then $\frac{p}{q},\frac{p'}{q'}\in \cT_{k-1}$.  Line 1 of
Lemma~\ref{lemma:cfracs} implies that $t=\frac{ap+p'}{aq+q'}$.  Line 2
of Lemma~\ref{lemma:cfracs} applied to the convergents $\frac{p}{q}$
and $t$ implies that
  \begin{equation*}
\left|\frac{p}{q}-t\right|=\frac{1}{q(aq+q')}\le \frac{1}{a}.
  \end{equation*}
Theorem \ref{thm:qxi} provides the existence of a universal constant
$B_1>1$ such that the radii of the excluded intervals about finite
elements of $\cT_{k-1}$ are at least $\zr:=
B^{-1}_1C^{-8}H_{k-1}^{-18}$.  It follows that if $a^{-1}<\zr$, then
$t$ is in the excluded interval about $\frac{p}{q}$.  Since this is
not the case, it follows that $a^{-1}\ge \zr$, that is, $a\le
\zr^{-1}$.  Thus $\height(t)\le (\zr^{-1}+1)H_{k-1}\leq
A_1C^8H_{k-1}^{19}$, where $A_1$ is a universal constant which may be
taken to be $B_1+1$. We conclude that we may set \[ H_k :=
A_1C^8H_{k-1}^{19}.\] We obtain by induction \[ H_k =
(A_1C^8)^{1+19+19^2+\ldots + 19^{k-1}}\cdot (A_0C^8)^{19^k}\] where
the constants $A_0, A_1$ are universal.

Proceeding in this way, we eventually construct $\cT_N$.
Theorem~\ref{thm:short_cf} implies that if $t=[a_0,\dotsc,a_n]$ and
$\zs_F(t)=t$, then $n\le N$.  Thus $t\in \cT_N$, and we have the
desired bound on $\height(t)$.  In particular, the height of a cusp
$t$ fixed by $\zs_F$ is at most 
  \begin{equation*}
H_N<(A_0A_1)^{19^N}(C^8)^{1+\cdots
+19^N}<(A_0A_1)^{19^N}(C^8)^{19^{N+1}/18}<(A_0A_1C^9)^{19^N}
=(AC^9)^{19^N},
  \end{equation*}
where $A$ is an absolute constant and $N=N(C,\cH)$ is the constant
from Theorem \ref{thm:short_cf}. 

This proves Theorem~\ref{thm:finite}.
\end{proof}

\section{The slope function $\mu_F$ is surjective or trivial}
\label{secn:surjective}
\nosubsections

This brief section proves a foundational result using part of the
correspondence on moduli space associated to $F$.  We remark that
$\mu_F$ can indeed be trivial; see \cite{BEKP}.

\begin{thm}[The slope function $\mu_F$ is surjective or trivial]
\label{thm:trivial_or_surjective} Suppose $F: (S^2,P) \to (S^2,P)$ is
an arbitrary Thurston map on the standard square pillowcase with
$P=P_F$. Then $\mu_F$ is surjective on $\overline{\bbQ}$ if and only if it is not
identically the constant function $\nonslope$.  Equivalently, the
extension of $\sigma_F$ to $\bbH \cup \overline{\bbQ}$ is surjective
on cusps if and only if $\sigma_F$ is not a constant function.
\end{thm}

\begin{proof}
Necessity follows from \cite[Theorem 5.1]{kps}. We now prove sufficiency. Suppose $\mu_F$ is nontrivial. Since $\mu_F(s)=-1/\sigma_F(-1/s)$, it suffices to show $\sigma_F$ is surjective on cusps of $\bbH$.  

Let us here for convenience denote by $G=\text{PMod}(S^2,P)$.  For
each liftable $h \in G_F$ and $\tau \in \bbH$, we have the
functional equation 
\[\sigma_F(h.\tau)=\phi_F(h).\sigma_F(\tau).\] 
It follows that we obtain the commutative diagram below, much as in
\cite[Figure 2]{kps}; the space $\cW$ is $\bbH/G_F$; the space
$\mathcal{Z}$ is $\bbH/\phi_F(G_F)$; the vertical map $\lambda$ is the usual holomorphic universal covering; the space $\cM$ is $\bbH/G$; the
maps $X$ and $\overline{\sigma_F}$ are holomorphic; the remaining maps
are holomorphic covering maps.  \[ \xymatrix{ \bbH \ar[dr] \ar[rrr]^{\sigma_F}
\ar[ddd]_\lambda &&& \bbH\ar[ddd]^\lambda \ar[ddl] \\ & \mathcal{W} \ar[ddl]^Y
\ar[dr]^{{\overline{\sigma}_F}}\ar@/ _1pc/[ddrr]_X&&\\ && \mathcal{Z}
\ar[dr] & \\ \cM &&& \cM.  } \] Since $\mu_F$ is nontrivial, $\phi_F$
is nontrivial, and $\sigma_F$ and $\overline{\sigma}_F$ are
nonconstant. Since $[G:G_F]$ is finite (\cite[Prop. 3.1]{kps}),
$\mathcal{\cW}$ is isomorphic to a compact Riemann surface punctured
at finitely many points, so the map $\overline{\sigma}_F: \mathcal{W}
\to \mathcal{Z}$ is surjective on ends.

Now suppose $t' \in \overline{\bbQ}$ is a cusp of the upper right copy
of $\bbH$. Let $e'_{\mathcal{Z}}$ be its image cusp in $\cZ$.  The
aforementioned surjectivity yields $t \in \overline{\bbQ}\subseteq \partial
\bbH$ at upper left and $e_{\mathcal{W}} \in \partial \mathcal{W}$
with $t \mapsto e_{\mathcal{W}} \mapsto e_\mathcal{Z}$. The cusps
$\sigma_F(t)$ and $t'$ project to the same point $e_\mathcal{Z} \in
\partial \mathcal{Z}$.  Hence there exists $h \in G_F$ with
$\phi_F(h).\sigma_F(t)=t'$.  By the functional equation, we have then
$\sigma_F(h.t)=\phi_F(h).\sigma_F(t)=t'$ as required.
\end{proof}

\section{How size controls geometry}
\label{secn:size}\nosubsections

In this section, we suppose $\mathbf{D}$ is a presentation diagram and
$F=F_{\bD}$ the corresponding NET map.  We further assume
$\geomsize(\mathbf{D}) \leq C$. We maintain $G=\text{PMod}(S^2,P)$.

Our first elementary observation is that $\deg(F)\leq C^2$; we use this repeatedly without mention in making further estimates. 

This brings us to the following result.

\begin{prop}
\label{prop:index_bound}
The index of the pure liftables satisfies 
  \begin{equation*}
[G:G_F] \leq \frac{2}{3}\deg(F)^3
\prod_{\substack{p|2\deg(F)\\p\text{ prime}}}(1-p^{-2})\leq\frac{1}{2}\deg(F)^3.
  \end{equation*}
Moreover, $G_F$ contains the projectivized principal congruence
subgroup $ \overline{\Gamma}(2\deg(F))$.
\end{prop}
  \begin{proof} Proposition 3.4 of \cite{fpp2} provides an estimate
for the index of the modular group liftables for $F$ in the modular
group of $F$.  That proof and the fact that the index of the pure
modular group in the modular group is 24 yield the first
inequality and the final statement. The second inequality comes from
the product factor for the prime 2.
\end{proof}

\begin{prop}[Small cusps mapping to each cusp in moduli space]
\label{prop:simple_input_nontrivial_output} Suppose $\sigma_F$ is not
a constant map. Let $\mathcal{E}:=\{0/1, 1/0, 1/1\} \subset
\overline{\mathbb{Q}}$.  Then for each end $e \in \cE$, there exists a
cusp $t_e$ with $\sigma_F(t_e)\in G.e$ and $\height(t_e) \leq
\deg(F)$.  
\end{prop}

\begin{proof} Let $e\in \cE$.  Theorem~\ref{thm:trivial_or_surjective} implies that there exists a cusp $t$ with $\sigma_F(t)=e$. 
Moreover, if $h \in G_F$ then 
  \begin{equation*}
\zs_F(h.t)=\phi_F(h).\zs_F(t)=\phi_F(h).e\in G.e.
  \end{equation*}
It remains to show that we can find simple cusps in $G_F.t$.

Proposition \ref{prop:index_bound} asserts that the group $G_F$ of
liftable pure modular group elements contains the projectivized
principal congruence subgroup $\overline{\Gamma}(2\deg(F))=\ker\pi$
where $\pi: \mathrm{PSL}_2(\Z) \to \mathrm{PSL_2}(\Z/2\deg(F)\Z)$.
Thus there exists a finite set $\{g_1, g_2, \ldots, g_n\}$ of coset
representatives for $\overline{\Gamma}(2\deg(F))$ in
$\mathrm{PSL}_2(\Z)$ each of which has entries in $\{0,\pm
1,\dotsc,\pm \deg(F)\}$.  Since $\mathrm{PSL}_2(\Z)$ acts transitively
on $\overline{\bbQ}$, it follows that the orbit $G_F.t$ contains an
element of the set $\{g_j.\frac{1}{0} | j=1,\dotsc, n\}$.  Since the
entries of $g_j$ lie in $\{0,\pm 1,\dotsc,\pm \deg(F)\}$ we conclude
that it is possible to find $t_e$ such that $\height(t_e) \leq \deg(F)$.

This proves Proposition \ref{prop:simple_input_nontrivial_output}.
\end{proof}

The proof of the next result uses the geometry of the dynamical plane,
so it is initially formulated in terms of slopes, as opposed to
cusps.

\begin{prop}[Linear height distortion]
\label{prop:linear_height_distortion} If $s$ is a slope such that
$s':=\mu_F(s)\neq \nonslope$, then $\height(s')\leq L\cdot \height(s)$
where $L=125C$. Thus for cusps $t$ such that $\zs_F(t)\in
\overline{\bbQ}$, we have correspondingly that $\height(\sigma_F(t))
\leq L\cdot \height(t)$.
\end{prop}

It is easy to see that such an upper bound must be at least linear in $C$. 

\begin{proof} Recall the algorithm from \cite[Section 5]{cfpp} for
evaluating $\zm_F$ given the presentation diagram $\mathbf{D}$.  Let $\zl_1$
and $\zl_2$ be the integer lattice vectors determined by $\mathbf{D}$.  Let
$\zL_1$ be the lattice generated by $\zl_1$ and $\zl_2$.  Let $F_1$ be
the parallelogram displayed in $\mathbf{D}$.  Let $\zG_1$ be the group
generated by $x\mapsto 2\zl-x$ for $\zl\in \zL_1$, so that $F_1$ is a
fundamental domain for $\zG_1$.  We tile $\bbR^2$ with the
$\zG_1$-translates of $F_1$.  The connected components of the union of
the $\zG_1$-translates of the green arcs in $F_1$ are called spin
mirrors.

Now let $s\in \overline{\bbQ}$ such that $s':=\zm_F(s)\ne \odot $.
Suppose that $s=\frac{p}{q}$ and $s'=\frac{p'}{q'}$, where $p$, $q$
and $p'$, $q'$ are two pairs of relatively prime integers.  Here we
briefly recall how we compute $s'$.  Let $d=d(s)$ be the local
covering degree associated to slope $s$. We choose an appropriately
generic point $v\in \bbR^2$, and we let $w=v+2d(q,p)$.  We imagine a
photon starting at $v$ and beginning to travel in a line toward $w$
(hence in the direction of $(q,p)$).  It travels toward $w$ until it
hits a spin mirror.  At the spin mirror, it rotates $180$ degrees
about the midpoint of the spin mirror and then proceeds in the
direction of $-(q,p)$.  It continues in this way, traveling in the
direction of either $(q,p)$ or $-(q,p)$.  Every time that it hits a
spin mirror, it rotates $180$ degrees about the midpoint of the spin
mirror and reverses direction.  Let $w'$ be the point the photon
reaches after traveling a distance $|v-w|$ along the line segments
with direction $\pm (q,p)$ (hence not counting spins).  Then $v-w'$ is
an integer multiple of $q'\zl_1+p'\zl_2$.  We use this description of
$p'$ and $q'$ to estimate their sizes.

Let $S$ be the line segment with endpoints $v$ and $w$.  By
definition, the distance that the photon travels while traveling
parallel to $S$ is the length of $S$.  The distance that it travels
during spins is at most the sum of the lengths of the spin mirrors
which $S$ meets.

Next let $||(x,y)||=\max\{|x|,|y|\}$ for every $(x,y)\in \bbR^2$.  Let
$h=\height(s)=||(q,p)||$ and $h'=\height(s')=||(q',p')||$.  Let
$A$ be the $2\times 2$ matrix whose columns are $\zl_1$ and $\zl_2$,
and let $\zF\co \bbR^2\to \bbR^2$ be the linear transformation defined
by $\zF(x)=A^{-1}x$.  Then $h'=||(q',p')||=||\zF(q'\zl_1+p'\zl_2)||$.
So $h'$ is at most the $||\cdot ||$-length $\ell (\zF(S))$ of $\zF(S)$
plus the sum of the $||\cdot ||$-lengths of the spin mirror
images under $\zF$ which $\zF(S)$ meets.

Now we estimate $\ell (\zF(S))$.  Using the fact that
$\text{det}(A)=\deg(F)$, we have that
  \begin{align*}
\ell (\zF(S)) & =||\zF(v-w)||=||\zF(2d(q,p))||
 =2d||A^{-1}\left[\begin{smallmatrix}q\\ p\end{smallmatrix}\right]||
\le \frac{4dCh}{\deg(F)}\le 4Ch.
  \end{align*}

It remains to estimate what the spin mirrors contribute to $h'$.  Note
that $\zF(F_1)$ is the rectangle with corners at $(0,0)$, $(2,0)$,
$(0,1)$ and $(2,1)$.  The map $\zF$ transforms the tiling of $\bbR^2$
by translates of $F_1$ to a tiling $\cT$ of $\bbR^2$ by translates of
this rectangle.  So the $||\cdot ||$-length of the image under $\zF$
of every spin mirror is at most 4.  The midpoint of each of these spin
mirror images is a lattice point in $\bbZ^2$.

We finally estimate the number of lattice points in $\bbZ^2$ which lie
in tiles of $\cT$ which meet $\zF(S)$.  We may choose $v$ so that
$\zF(S)$ contains no lattice points and the coordinates of the
endpoints of $\zF(S)$ are not integers.  Let $X$ be the union of all
closed rectangles with width 2, height 1 and corners in $\bbZ^2$ which
meet $\zF(S)$.  (Not all of these are in $\cT$.)  Let $a$,
respectively $b$, be the absolute value of the difference of the $x$,
respectively $y$, coordinates of the endpoints of $\zF(S)$.  So $\ell
(\zF(S))=||(a,b)||$.  An induction argument based on the number of
points in $\zF(S)$ which have an integer coordinate proves that the
number of lattice points in $X$ is $2a+4b+8$.  See
Figure~\ref{fig:countlatpts}, where $2a$ lattice points are drawn with
squares, $4b$ lattice points are drawn with circles and 8 lattice
points are drawn with $\times $'s.  So the spin mirror contribution to
$h'$ is at most $(6\ell(\zF(S))+8)\cdot 4$.

  \begin{figure}
\centerline{\includegraphics{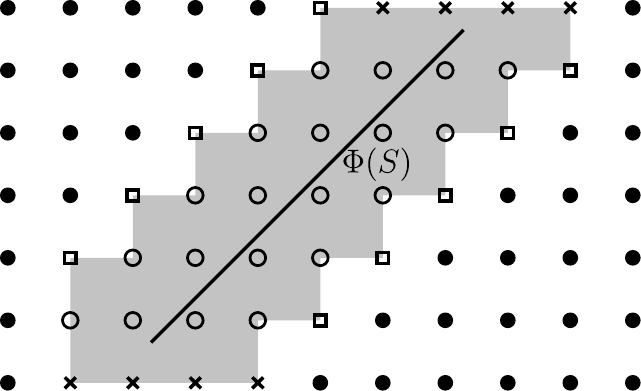}}
\caption{Counting the number of lattice points near $\zF(S)$}
\label{fig:countlatpts}
  \end{figure}

Assembling the above and using the fact that $C\ge \sqrt{2}>7/5$, we
have that
  \begin{align*}
h' & \le \ell (\zF(S))+(6\ell (\zF(S))+8)4\le 25\ell (\zF(S))+32
 \le 25\cdot 4Ch+32\le 125Ch.
  \end{align*}

This proves Proposition~\ref{prop:linear_height_distortion}.
\end{proof}
Below, $d_\bbH$ denotes hyperbolic distance. 

\begin{thm}[Bounded displacement of basepoint]
\label{thm:bounded_displacement}
We have 
\[ d_{\bbH}(i, \sigma_F(i))\leq 30\log(5C)+\log16\lesssim \log(C);\]
the implicit constant is universal. 
\end{thm}
Easy examples show that this displacement can be at least a constant times $\log(C)$, so the form of the estimate is sharp.

\begin{proof} Proposition \ref{prop:simple_input_nontrivial_output}
provides cusps $t_0=p_0/q_0$, $t_1=p_1/q_1$ of heights $h_0, h_1$ at
most $\deg(F)\leq C^2$ such that their images $t_0'=p_0'/q_0'$,
$t_1'=p_1'/q_1'$ under $\sigma_F$ lie in the cusp orbits
$\overline{\Gamma}(2).0/1$ and $\overline{\Gamma}(2).1/1$,
respectively. Thus neither $t_0'$, $t_1'$ is the point at infinity.
Setting $L:=125C$ to be the constant from Proposition
\ref{prop:linear_height_distortion}, we conclude $h_j':=\height(t_j')
\leq LC^2$.  This implies that $|t_j'| \leq LC^2$, $j=0,1$, and
$|t_0'-t_1'| \geq (LC^2)^{-2}$. Let $B:=(LC^2)^2=(5C)^6$.  Sacrificing
tightness of bounds for convenience of exposition, we conclude: \[
|t_0'|< B, \;\; |t_1'|< B, \;\; |t_0'-t_1'| \geq B^{-1}.\]

Let $B_j$ be the unique closed horoball tangent to $t_j$ for which the
horocycle $\partial B_j$ contains the basepoint $i=\sqrt{-1}$. On the
corresponding basepoint (square) torus $\bbC/\langle 1, i\rangle$, the
family of curves with slope $s_j:=-1/t_j$ has modulus
$m_j:=(p_j^2+q_j^2)^{-1}$. The height bound from the previous
paragraph gives $m_j \geq h_j^{-2}/2 \geq C^{-4}/2$. The Gr\"otzsch
inequality implies that on the torus $\bbC/\langle 1,
\sigma_F(i)\rangle$, the modulus $m_j'$ of the family of curves with
slope $s_j':=-1/t_j'$ satisfies $m_j' \geq \delta(s_j)m_j \geq
m_j/\deg(F)\geq C^{-6}/2$.  The locus in $\bbH$ where the modulus of
the curve family with slope $s_j'$ is exactly $m_j'$ is the circle
tangent to $t_j'$ of Euclidean radius $r_j'=\frac{1}{2m_j' (q_j')^2}$.
Using the inequality $(q_j')^2\geq 1$, it follows that $\sigma_F(i)$
lies inside the intersection of the closed horoballs tangent to $t_j'$
and of Euclidean radius $r_j'<C^6$. Again sacrificing tightness of
bounds for convenience of exposition, we conclude $\sigma_F(i) \in
K:=B_0' \cap B_1'$, where $B_j'$ is the closed horoball tangent to
$t_j'$ of Euclidean radius $B$.

In this paragraph, we bound the hyperbolic diameter of $K$.  Recall
the horoballs $B_j'$ have common Euclidean radius $B$. It follows that
the maximum value of $K$ occurs when $|t_0'-t_1'|=B^{-1}$. Since
$x\mapsto B^{-1}x$ is a hyperbolic isometry, to compute
$\diam_{\bbH}(K)$ we may assume that the horoballs $B'_j$ have radius
1 and that $|t_0'-t_1'|=B^{-2}$.  An easy calculation gives
$\diam_{\bbH}(K) = \log((1+A)/(1-A))$, where $A=\sqrt{1-B^{-4}/4}$.
Since $(1+A)/(1-A)<4/(1-A^2)$ for $0<A<1$, we find that
$\diam_{\bbH}(K)<\log(16B^4)$.

In this paragraph, we bound the hyperbolic distance from $K$ to $i$.
Since $|t'_j|<B$ and $B'_j$ has radius $B$, we have that $Bi\in B'_j$.
Hence the path which runs vertically from $i$ to $Bi$ joins $i$ and
$K$.  Hence the hyperbolic distance from $K$ to $i$ is at most
$\log(B)$. 

Combining the previous two paragraphs, we conclude $d_{\bbH}(i,\sigma_F(i)) \leq d_{\bbH}(i, K) + \diam_{\bbH}(K) \leq \log(B^5)+\log16$, as required. 
\end{proof}

\section{Fixed cusps have short continued fraction expansions}
\label{secn:pf_thm_short_cf}\nosubsections

We prove Theorem \ref{thm:short_cf} in this section.  The idea is
quite simple. Let $\gamma$ be a geodesic from the basepoint
$\sqrt{-1}=i \in \bbH$ to the boundary of the unit modulus horoball
$B$ tangent to a fixed cusp $t$. Theorem
\ref{thm:bounded_displacement} says the size bound $C$ controls the
hyperbolic distance $d_{\bbH}(i,\sigma_F(i))$. We will prove that if
the continued fraction expansion of $t$ is long, the hyperbolic length
of $\sigma_F(\gamma)$ is much smaller than that of $\gamma$. Since $t$
is fixed, $\sigma_F(B)\approx B$. But this is impossible.

\begin{proof}
We begin with an elementary observation about the hyperbolic geometry of a regular ideal quadrilateral $Q$. Consider Figure \ref{fig:fareyquad}, which shows $Q$ outlined in bold lines in the disk model of the hyperbolic plane. 
  \begin{figure}
\centerline{\includegraphics[width=2in]{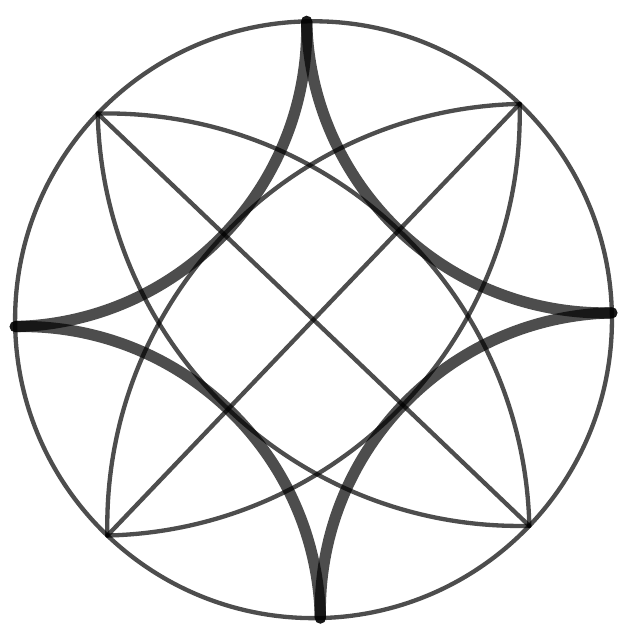}}
\caption{The distance between opposite sides of a regular ideal
quadrilateral is the constant $a$. Any geodesic joining opposite sides
contains a segment of length at least $a$ contained in the union of
the two $\frac{a}{2}$-neighborhoods of the two perpendiculars. }
\label{fig:fareyquad}
  \end{figure}
Let $a:=4\tanh^{-1}(\sqrt{2}-1)\approx 0.88137$ be the distance
between opposite sides of $Q$. Let $X_Q \subset Q$ be the union of the
two closed hyperbolic $\frac{a}{2}$-neighborhoods of the two geodesics perpendicular to each pair of opposite sides.  If $\gamma$ is a hyperbolic geodesic meeting opposite sides of $Q$, then the intersection of $\gamma$ with $X_Q$ is a segment of length at least $a$.  

The moduli space $\bbH/\overline{\Gamma}(2)$ is a quotient of $Q$
obtained by side-pairings with isometries. Let $X \subset
\bbH/\overline{\Gamma}(2)$ be the image of $X_Q$ under this natural
projection, and $\widetilde{X}:=\lambda^{-1}(X)$ its lift to the upper
half-plane, where $\zl\co \bbH\to \bbH/\overline{\zG}(2)$ is the
canonical quotient map. Since $X$ is compact, and since
$||d\sigma_F(x)||<1$ is a continuous function on a finite cover $\cW$
of moduli space $\bbH/\overline{\Gamma}(2)$ (from Section
\ref{secn:surjective} or \cite[Lemma 5.2]{DH}), there is a real number
$c$ such that
\[ \max\{ ||d\sigma_F(x)|| : x \in \widetilde{X}\} \leq c<1.\]
Composing $F$ by an element of $\Mod(S^2, P)$ changes $\sigma_F$ by
composition with an isometry, so $c$ depends only on the modular group
Hurwitz class $\cH$ of $F$ as defined in \cite{fpp2}.

Now suppose $t=\sigma_F(t)$ is a fixed cusp.  We may assume $t \neq 1/0$.  Let
$t=[a_0,\dotsc,a_n]$, so that the nonnegative integer $n$ is the
length of this continued fraction.  Let $t:=p/q$ for relatively prime
integers $p$, $q$ and let $B$ be the closed horoball tangent to $t$ of
radius $r:=\frac{1}{2q^2}$. The Gr\"otzsch inequality implies that
$\sigma_F(B)$ is contained in the closed horoball $B'$ tangent to $t$
of radius $r':=\deg(F)\cdot r$. To ease notation in the displays
below, we write $d(\cdot, \cdot)$ for hyperbolic distance.  Note that
$d(\partial B, \partial B')=\log\deg(F)$.

Let $\gamma$ be the closure of the segment of the geodesic from the basepoint $i$ to $\partial B$, and $\ell(\gamma)$ its hyperbolic length. Let $\gamma':=\sigma_F(\gamma)$. Then $\gamma'$ joins $\sigma_F(i)$ to $B'$ and so by the previous paragraph and the triangle inequality 
\[ d(i,B) =\ell(\gamma) \leq d(i,\sigma_F(i)) + \ell(\gamma')+\log\deg(F).\]
We will show 
\begin{equation}
\label{eqn:additive_contraction}
 \ell(\gamma')\leq \ell(\gamma)-a \cdot (n-1)\cdot(1-c).
 \end{equation}
Assuming inequality (\ref{eqn:additive_contraction}) for the moment, we conclude that 
\[ n \leq N:=1+\left(d(i,\sigma_F(i)\right)+\log \deg(F))\cdot (1-c)^{-1}a^{-1}.\] 
Appealing to Theorem \ref{thm:bounded_displacement}, we conclude 
\[ N \leq 1+\left(30\log(5C)+\log(16)+2\log(C)\right)\cdot (1-c)^{-1}a^{-1}\]
where the constant $c$ depends only on the Hurwitz class of $F$ and
the constant $a$ is universal.  This gives the conclusion of 
Theorem~\ref{thm:short_cf}.
 
We now establish inequality (\ref{eqn:additive_contraction}).  
Consider Figure \ref{fig:fareystrip}. This draws the hyperbolic
plane in the band model, and the geodesic through $i$ and asymptotic
to $t$ as the horizontal line of symmetry of the band.  The triangles drawn represent, combinatorially, the Farey tiling.  
  \begin{figure}
\centerline{\includegraphics[width=4in]{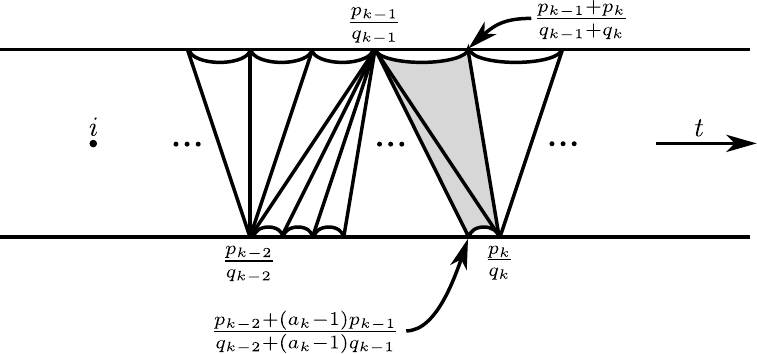}}
\caption{The geodesic $\gamma$ joins $i$ and $t$ and is the horizontal line containing the three triples of dots at the center of the band. It 
must join opposite sides of at least $n-1$ distinct Farey quadrilaterals, one of which is shaded.}
\label{fig:fareystrip}
  \end{figure}
The convergents of $t=[a_0,\dotsc,a_n]$ are
$\frac{p_0}{q_0},\dotsc,\frac{p_n}{q_n}$.  The quadrilaterals, such as
the one shaded, whose interior geodesic joins two consecutive
convergents of $t$ are Farey quadrilaterals.  The horoball $B$ is
chosen so that its interior is disjoint from every Farey quadrilateral
which does not have a vertex at $t=\frac{p_n}{q_n}$.  So $\zg$ must
join opposite sides of every Farey quadrilateral corresponding to
convergent pairs
$(\frac{p_0}{q_0},\frac{p_1}{q_1}),\dotsc,
(\frac{p_{n-2}}{q_{n-2}},\frac{p_{n-1}}{q_{n-1}})$.  Let
$\gamma_j$, $j=1, \ldots, n-1$ be the subsegments of $\gamma$
comprising the intersections of $\gamma$ with these Farey
quadrilaterals, so that $\ell(\gamma_j) \geq a$, and let $\gamma_0:=\gamma - \cup_j \gamma_j$.  Let $\gamma_j':=\sigma_F(\gamma_j)$ and $\gamma_0'=\sigma_F(\gamma_0)$. Then 
  \begin{align*}
\ell(\gamma') & \le\ell(\gamma_0')+\sum_j \ell(\gamma_j') 
\leq \ell(\gamma_0)+\sum_jc\cdot \ell(\gamma_j)\\
 & =\ell(\gamma_0)+\sum_j(\ell(\gamma_j)-(1-c)\ell(\gamma_j))
   \leq \ell(\gamma)-(n-1)(1-c)a,
  \end{align*}
as required.
\end{proof}

\section{Quantitative excluded intervals}
\label{secn:pf_thm_qxi}

Here we prove Theorem \ref{thm:qxi}

\begin{proof} We begin by recalling notation.  Let $s$ be a slope, and
let $t$ be the cusp $t=-1/s$.  Temporarily abusing notation, we set
$s':=\mu_F(s), t':=\sigma_F(t)$ even if $s'=\nonslope$, which
corresponds to $t'=\nonslope$. Thus if $s' \neq \nonslope$ then
$t'=-1/s'$.  We write $h=\height(s)=\height(t)$ and
$\delta(s)=c(s)/d(s)$, the corresponding multiplier. We have
$s'=\nonslope \iff \delta(s)=0$.  The proof of Theorem \ref{thm:qxi}
divides into cases not only along the lines of the two parts of its
conclusion, but also according to whether or not the cusp $t$ falls
into one of a few exceptional cases determined by the condition $t'\in
\{1/0, t, \nonslope\}$.  We will treat these various cases using Lemma
\ref{lemma:easy} and \ref{lemma:hard}. The following table summarizes
the cases and bounds obtained.\\
\begin{center}
\begin{tabular}{ l | l | l }
& $t=1/0$ & $t\neq 1/0$\\ \hline 
  $t' \not\in\{1/0, t,\odot \}$ & Lemma \ref{lemma:easy}, $R\sim C^2$ &
Lemma \ref{lemma:easy}, $\rho \sim C^{-2}h^{-2}$\\
  $t'=1/0$ & Lemma \ref{lemma:hard}, $R \sim C^8$ & Lemma \ref{lemma:easy}, $\rho\sim C^{-2}h^{-2}$ \\
  $t=t'\neq 1/0$ &  not applicable & Lemma \ref{lemma:hard}, $\rho \sim C^{-8}h^{-18}$\\
 $ t'=\nonslope$ & Lemma \ref{lemma:hard}, $R \sim C^8$ & Lemma \ref{lemma:hard}, $\rho \sim C^{-8}h^{-18}$.
\end{tabular}
\end{center}
The implicit constants are universal.

Corollary 7.3 of \cite{P} implies the following theorem.

\begin{thm}\label{thm:fixedpt} Suppose that $s=\frac{p}{q}$ and that
$s'=\frac{p'}{q'}$ with $\gcd(p,q)=\gcd(p',q')=1$.  Then the set of
real numbers $x$ defined by
  \begin{equation}\linnum\label{eqn:fixedpt}
\deg(F)\left|px+q\right|<\left|p'x+q'\right|
  \end{equation}
does not contain a cusp fixed by $\sigma_F$.
\end{thm}

We use this theorem to prove the following lemmas.

\begin{lemma}[Excluded intervals for nonexceptional
cusps]\label{lemma:easy} Suppose that $t'\notin \{t,\odot \}$ and that
$\frac{u}{v}\in \bbQ$ is a cusp fixed by $\zs_F$.  
\begin{enumerate}
  \item If $t=\frac{1}{0}$, then
$\frac{u}{v}\notin I_t:=(-\infty ,-R)\cup (R,\infty )$, where
$R=100C^2 \sim C^2$.
  \item If $t\ne \frac{1}{0}$, then
$\frac{u}{v}\notin I_t:=(t-\zr,t+\zr)$, where $\zr=100^{-1}C^{-2}h^{-2}\sim C^{-2}h^{-2}$.
\end{enumerate}
\end{lemma}
  \begin{proof} Suppose that $s=\frac{p}{q}$ and $s'=\frac{p'}{q'}$
with $\gcd(p,q)=\gcd(p',q')=1$.  So $t=-\frac{q}{p}$ and
$t'=-\frac{q'}{p'}$.  Theorem~\ref{thm:fixedpt} states that the set
of real numbers $x$ which satisfy the inequality in line \ref{eqn:fixedpt}
does not contain a cusp fixed by $\zs_F$.

First suppose that $t=\frac{1}{0}$, so that $p=0$, $|q|=1$ and $h=1$.
Then $p'\ne 0$ because $t'\ne t$.  We take $L=125C$ as in
Proposition~\ref{prop:linear_height_distortion}.  We note that
$99C>125$ because $C\ge \sqrt{2}$.  So if $|x|>100C^2$, then
  \begin{equation*}
|p'x+q'|\ge |p'x|-|q'|\ge |x|-|q'|> 100C^2-L= 99C^2+C^2-125C>C^2\ge \deg(F).
  \end{equation*}
This shows that $x$ satisfies line \ref{eqn:fixedpt}, proving statement
1.

Now to prove statement 2, suppose that $t\ne \frac{1}{0}$,
equivalently, $p\ne 0$.  We rewrite line \ref{eqn:fixedpt} as
$|\zv(x)|<\deg(F)^{-1}$, where $\zv(x)=\frac{px+q}{p'x+q'}$.  We want
to choose the real number $\zr$ in the statement of the lemma so that
the interval $(t-\zr,t+\zr)$ lies in the interval about $t$ with
endpoints $\zv^{-1}(\deg(F)^{-1})$ and $\zv^{-1}(-\deg(F)^{-1})$.  It
suffices to choose $\zr$ so that
  \begin{equation*}
\zr\le \min\{|\zv^{-1}(\pm \deg(F)^{-1})-t|\}.
  \end{equation*}
To estimate this minimum, we set $r=\pm \deg(F)^{-1}$.  Then
  \begin{align*}
|\zv^{-1}(r)-t| & =\left|\frac{q'r-q}{-p'r+p}+\frac{q}{p}\right|
=\frac{|pq'-qp'||r|}{|-p'r+p||p|}\ge \frac{1}{|-p'+pr^{-1}||p|}\\
 & \ge \frac{1}{(Lh+hC^2)h}=(L+C^2)^{-1}h^{-2}\ge 100^{-1}C^{-2}h^{-2}.
  \end{align*}

This proves Lemma~\ref{lemma:easy}.

\end{proof}

\begin{lemma}[Excluded intervals for exceptional
cusps]\label{lemma:hard} Suppose that $t'\in \{t,\odot \}$ and if
$t'=t$, then $\zd(s)\ne 1$.  Suppose $\frac{u}{v}\in \bbQ$ is a cusp
fixed by $\zs_F$.  
\begin{enumerate}
  \item If $t=\frac{1}{0}$, then $\frac{u}{v}\notin I_{1/0}:=(-\infty
,-R)\cup (R,\infty )$, where $R\sim C^8$.
  \item If $t\ne \frac{1}{0}$, then $\frac{u}{v}\notin
I_t:=(t-\zr,t)\cup (t,t+\zr)$, where $\zr\sim C^{-8}h^{-18}$.
\end{enumerate}
\end{lemma}
  \begin{proof} We first prove statement 1.  Suppose that
$t=\frac{1}{0}$.

Proposition~\ref{prop:simple_input_nontrivial_output} implies that
there exists a slope $\frac{P}{Q}$ such that $P\ne 0$, $\gcd(P,Q)=1$,
$\zm_F(\frac{P}{Q})\notin \{\frac{0}{1},\odot \}$ and
$\height(\frac{P}{Q})\le \deg(F)$.  So $-\frac{Q}{P}$ is a finite
cusp such that $\zs_F(-\frac{Q}{P})$ is a finite cusp.  

Now we introduce the stabilizer of the cusp $\frac{1}{0}$.  Let $T$ be
the primitive, positive Dehn twist about the curve $\zg$ of slope
$\frac{0}{1}$.  Suppose that $\zg$ has $c$ essential nonperipheral
preimages, each mapping by degree $d$, so that
$\zd(\frac{0}{1})=\frac{c}{d}$.  Then, up to homotopy,
  \begin{equation*}
T^d\circ F=F\circ T^c.
  \end{equation*}
The map on $\bbH$ induced by $T$ is $z\mapsto z+2$, so
  \begin{equation*}
\zs_F(z+2d)=\zs_F(z)+2c.
  \end{equation*}
Because $\zd(s)\ne 1$ if $t'=t$, it follows that $\zs_F$ maps
$-\frac{Q}{P}-2nd$ to a different finite cusp for all but at most one
integer $n$.

Next let $n\in \bbZ$.  The previous paragraph shows that we may argue
as in the proof of statement 2 of Lemma~\ref{lemma:easy}, replacing
$-\frac{q}{p}$ there by $-\frac{Q}{P}-2nd$ for all but at most one
value of $n$.  Let $\frac{P'}{Q'}=\zm_F(\frac{P}{Q})$ with
$\gcd(P',Q')=1$.  Then $P'\ne 0$ and
$\zs_F(-\frac{Q}{P}-2nd)=-\frac{Q'}{P'}-2nc$.

Our next goal is to find a positive integer $N$ so that if $|n|\ge N$,
then we can find an excluded interval about $-\frac{Q}{P}-2nd$ with
radius greater than $d$.  The union of these intervals over such
sufficiently large integers $n$ will give us $I_{1/0}$.  So let
  \begin{equation*}
\zv(x)=\frac{Px+Q+2ndP}{P'x+Q'+2ncP'}.
  \end{equation*}
Arguing as in the proof of statement 2 of Lemma~\ref{lemma:easy}, we
want
  \begin{equation*}
\min\left\{\left|\zv^{-1}(\pm \deg(F)^{-1})+\frac{Q}{P}+2nd\right|\right\}>d.
  \end{equation*}
Set $r=\pm \deg(F)^{-1}$.  Then assuming that $\zs_F$ does not fix
$-\frac{Q}{P}-2nd$,
  \begin{align*}
\left|\zv^{-1}(r)+\frac{Q}{P}+2nd\right| & 
=\left|\frac{(Q'+2ncP')r-Q-2ndP}{-P'r+P}
+\frac{Q+2ndP}{P}\right|\\
 & =\left|\frac{(PQ'-P'Q+2n(c-d)PP')r}{(-P'r+P)P}\right|\\
 & =
\left|\frac{Q'(P')^{-1}-QP^{-1}+2n(c-d)}{-1+P(P')^{-1}r^{-1}}\right|\\
 & \ge \frac{2|n||c-d|-\left|Q'(P')^{-1}\right|-\left|QP^{-1}\right|}
{|P||P'|^{-1}|r|^{-1}+1}\\
  & > \frac{2|n|-L\deg(F)-\deg(F)}{\deg(F)^2+1}\\
 & > \frac{|n|-(L+1)\deg(F)}{\deg(F)^2+1}.
  \end{align*}
We want this to be at least $d\le \deg(F)\le C^2$.  So we want
  \begin{equation*}
|n|\ge(C^4+1)C^2+(125C+1)C^2.
  \end{equation*}
So setting
  \begin{equation*}
N=(C^4+125C+2)C^2,
  \end{equation*}
if $|n|\ge N$, then we have an excluded interval about
$-\frac{Q}{P}-2nd$ with radius greater than $d$.  Note that if $\zs_F$
fixes $-\frac{Q}{P}-2nd$, then the right side of the large display
above is 0.  Since it is not 0, $\zs_F$ does not fix
$-\frac{Q}{P}-2nd$ if $|n|\ge N$.  So the union of these excluded
intervals contains $I_{1/0}$ by taking
  \begin{equation*}
R=\max\left\{\left|\frac{Q}{P}+2Nd\right|,\left|\frac{Q}{P}-2Nd\right|\right\}
\le \left|\frac{Q}{P}\right|+2Nd
 \le C^2+2NC^2\sim C^8.
  \end{equation*}
This proves statement 1.

Now we consider statement 2.  The strategy of this proof is to
conjugate $F$ so that $t$ moves to $\frac{1}{0}$ in order to reduce to
statement 1.  For this we construct a homeomorphism $\zv\co (S^2,P)\to
(S^2,P)$ induced by an element $M$ of $\text{SL}(2,\bbZ)$, viewing
$\text{PSL}(2,\bbZ)$ as a subgroup of the modular group of $F$.  We
choose the first column of $M$ to consist of $q$ and $p$.  To define
the second column of $M$, first suppose that $\frac{q}{p}\notin \bbZ$.
We express $\frac{q}{p}$ as a regular continued fraction.  The last
convergent of this continued fraction is $\frac{q}{p}$.  Let
$\frac{b}{a}$ be the previous convergent for integers $a$, $b$ with
$\gcd(a,b)=1$.  We temporarily choose the second column of $M$ to
consist of $b$ and $a$.  Then $\text{det}(M)=\pm 1$, as in statement 2
of Lemma~\ref{lemma:cfracs}.  We multiply the second column of $M$ by
$-1$ if necessary so that $\text{det}(M)=1$.  If $\frac{q}{p}\in
\bbZ$, then we simply take $b=\pm 1$ and $a=0$.  Let
$M=\left[\begin{smallmatrix}q & b \\ p & a
\end{smallmatrix}\right]\in \text{SL}(2,\bbZ)$.  The map on the square 
pillowcase induced by $M$ is $\zv$.

Now we set $G=\zv^{-1}\circ F\circ \zv$, a NET map.  The NET map $F$
is given by a diagram with size $C$.  Let $A$ be the defining matrix
for this diagram.  Then, as in Section 6 of \cite{fpp1}, a defining
matrix for $G$ is $M^{-1}AM$. We have that $\height(\frac{p}{q})=h$.
Since heights of successive convergents increase,
$\height(\frac{a}{b})\le h$.  It follows that there exists a
presentation diagram for $G$ with size at most $Ch^2$.

We are prepared to consider pullback maps on $\bbH$.  Let $\zs_F$,
$\zs_G$ and $\zs_\zv$ be the pullback maps of $F$, $G$ and $\zv$.
Then $\zs_G=\zs_\zv\circ \zs_F\circ \zs_\zv^{-1}$.  As in Section 6 of
\cite{fkklpps} or Proposition 4.1 of \cite{fpp2}, we have that 
  \begin{equation*}
\zs_\zv(z)=\frac{az+b}{pz+q}.
  \end{equation*}
So $\zs_\zv(-\frac{q}{p})=\frac{1}{0}$.  Hence $\zs_\zv$ transforms an
excluded interval for $F$ about $t$ to an excluded interval for $G$
about $\frac{1}{0}$.  Statement 1 applied to $G$ yields a value of
$R\lesssim C^8h^{16}$.  It suffices to choose $\zr$ so that it
is no larger than the distances from $\zs_\zv^{-1}(\pm R)$ to
$-\frac{q}{p}$.  We compute:
  \begin{align*}
\left|\zs_\zv^{-1}(\pm R)-(-\frac{q}{p})\right| 
  & =\left|\frac{q(\pm R)-b}{-p(\pm R)+a}+\frac{q}{p}\right| 
   =\frac{|qa-pb|}{|-p(\pm R)+a||p|}=\frac{1}{|-p(\pm R)+a||p|}\\
  & \ge \frac{1}{(|pR|+|a|)|p|}
  \gtrsim \frac{1}{(hC^8h^{16}+h)h}\sim C^{-8}h^{-18}
  \end{align*}

This proves statement 2.

\end{proof}
Lemmas \ref{lemma:easy} and \ref{lemma:hard} imply Theorem \ref{thm:qxi}. 
\end{proof}

\section{A sequence of NET maps with large obstruction slope heights }
\label{sec:largeratio}\nosubsections

This section is devoted to proving statement 2 of Theorem~\ref{thm:finite}.  

We begin by mentioning that since translations act trivially on
slopes, the choice of translation term--the vector circled in the
presentation diagram--does not affect the slope function or multiplier
function of a NET map $F$. A \emph{virtual} NET map presentation
diagram is such a diagram without such a choice of translation term.
Thus each virtual presentation diagram corresponds to a priori at most
four NET maps: we circle either 0, $\zl_1$, $\zl_2$ or $\zl_1+\zl_2$
in the diagram.  In low degree cases, such as the one we discuss here,
the number of corresponding NET maps might be less than four, since
some choices of translation term may yield maps with fewer than four
postcritical points, as required in the definition of NET map.  Thus
the set of non-Euclidean NET maps produced by a virtual NET map
presentation diagram consists either entirely of unobstructed maps, or
entirely of obstructed maps, each with a common obstruction.

We will construct an infinite sequence of degree 2 NET map virtual
presentation diagrams $\bD_n$ for integers $n\ge 0$.  It will be clear
from the construction that $n\le \text{geomsize}(\bD_n)\le n+2$.  We
will show that the non-Euclidean NET maps defined by $\bD_n$ with
$n\equiv 2\text{ mod } 3$ are obstructed.   Let $s_n$ be the slope of
the unique obstruction for such a diagram $\bD_n$ with respect to the
marking of $S^2$ determined by $\bD_n$.  We will then show that
$\height(s_n)$ grows roughly as $n^{\log_2(n)/2}$.  So the height of
the slope of the obstruction of $\bD_n$ grows faster than every
polynomial in the geometric size of $\bD_n$.

We will actually obtain a bit more.  The diagram $\bD_n$ is only a
virtual NET map presentation diagram because no element of $\zL_1$ is
circled.  We will find that circling three of the elements 0, $\zl_1$,
$\zl_2$ or $\zl_1+\zl_2$ obtains NET maps.  The fourth Thurston map
has only three postcritical points.  The dynamic portraits of these
three NET maps--invariants describing the dynamics and local degrees
on the set of critical points and their forward orbits--are mutually
inequivalent, so these three NET maps are mutually inequivalent.  In
Theorem~\ref{thm:conjugation2} we will show for $n\equiv 2\text{ mod
}3$ that the three equivalence classes of NET maps arising from
$\bD_n$ are independent of $n$.  

{\bf Remark.} Our analysis also shows that for each $n \geq 0$ with
$n\equiv 0,1\text{ mod }3$ that each NET map $F_n$ arising from
diagram $\bD_n$ is equivalent to one arising from $\bD_0$.  The computer program {\tt NETmap} says that each NET map in
$\bD_0$ is equivalent to a rational map.

\textbf{Definition of $\bD_n$.}  Let $n$ be a nonnegative integer.  We
define the virtual NET map presentation diagram $\bD_n$ so that
$\zl_1=(n,n+1)$ and $\zl_2=(n-2,n-1)$.  One verifies that the
determinant of the $2\times 2$ matrix whose columns are $\zl_1$ and
$\zl_2$ is 2 and that the parallelogram whose vertices are 0,
$2\zl_1$, $\zl_2$ and $2\zl_1+\zl_2$ contains $(n-1,n)$.  We choose
green line segments in this parallelogram so that only one is
nontrivial and its endpoints are $\zl_1$ and $(n-1,n)$.  This defines
$\bD_n$.  Clearly $n\le \text{geomsize}(\bD_n)\le n+2$.
Figure~\ref{fig:prendgm} shows $\bD_0$ and $\bD_2$.

  \begin{figure}
\centerline{\includegraphics{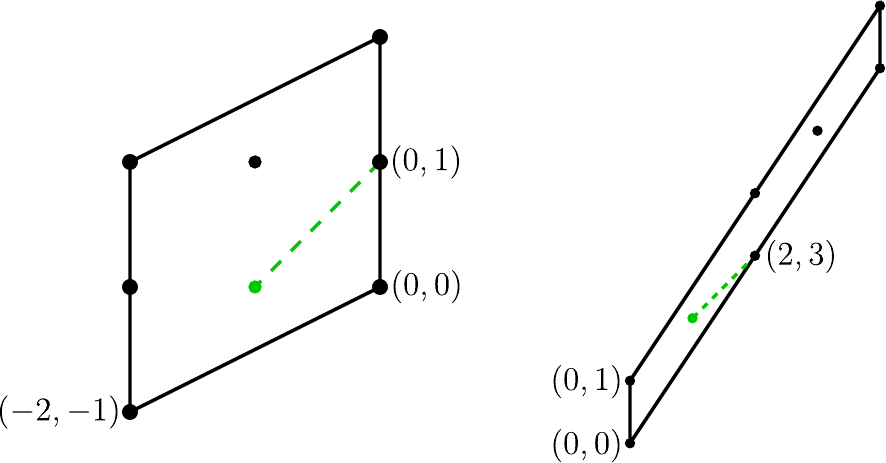}}
\caption{Virtual presentation diagrams $\bD_0$ at left and $\bD_2=M_2.\bD_0$ at right.}
\label{fig:prendgm}
  \end{figure}

\textbf{Twist relationship.}  There is a natural action of
$\text{SL}_2(\bbZ)$ on diagrams: just apply the linear map induced by
the matrix to the diagram; see \cite{fpp1}.  Here we find $M_n\in
\text{SL}(2,\bbZ)$ for which $\mathbf{D}_n=M_n.\mathbf{D}_0$.

One easily verifies that  
  \begin{equation*}
\left[\begin{matrix}-n+1 & n \\ -n & n+1 \end{matrix}\right]
\left[\begin{matrix}0 & -2 & -1 \\ 1 & -1 & 0 \end{matrix}\right]=
\left[\begin{matrix}n & n-2 & n-1 \\n+1  & n-1 & n
\end{matrix}\right].
  \end{equation*}
As in Theorem 1.2 of \cite{fpp1}, this means that $\bD_n=M_n.\mathbf{D}_0$ where 
  \begin{align*}\linnum\label{line:twistmatrix}
M_n:=&\left[\begin{matrix}-n+1 & n \\ -n & n+1 \end{matrix}\right]=
\left[\begin{matrix}1 & 0 \\ 0 & 1 \end{matrix}\right]+n
\left[\begin{matrix}-1 & 1 \\ -1 & 1 \end{matrix}\right]\\
=&\left(\left[\begin{matrix}1 & 0 \\ 0 & 1 \end{matrix}\right]+
\left[\begin{matrix}-1 & 1 \\ -1 & 1 \end{matrix}\right]\right)^n=
\left[\begin{matrix}0 & 1 \\ -1 & 2 \end{matrix}\right]^n.
  \end{align*}

\textbf{Dynamic portraits.}  We compute dynamic portraits of NET maps
as in Section 10 of \cite{fpp2}. Postcritical points which are not
critical are images in $S^2$ of 0, $\zl_1$, $\zl_2$ and $\zl_1+\zl_2$.
Let $\overline{0}$, $\overline{\zl_1}$, $\overline{\zl_2}$ and
$\overline{\zl_1+\zl_2}$ denote these images in $S^2$.  Let $\za$ and
$\zb$ denote the two critical points with $\zb$ postcritical.
Straightforward computations yield the following dynamic portraits for
$\bD_n$, identified by the parity of $n$ and which lattice point of
$\bD_n$ is circled.
\medskip

$n$ even:\qquad
\parbox{.7\linewidth}{
circle 0 (not NET): $\xymatrix{
\za \ar[r]^(.35){2} & \overline{\zl_1+\zl_2} \ar[r]^(.63){1} & \overline{0} 
  \ar@(dr,ur)[]_1 & &\zb \ar@(dr,ur)[]_2}$\\
circle $\zl_1$: $\xymatrix{\za \ar[r]^(.45)2 & \overline{\zl_2} 
\ar[r]^(.35){1} & \overline{\zl_1+\zl_2} \ar[r]^(.62){1} & 
\zb \ar@/_/[r]_2 & \overline{0} \ar@/_/[l]_1}$\\
circle $\zl_2$: $\xymatrix{\za \ar[r]^2 & \zb 
\ar[r]^(.35){2} & \overline{\zl_1+\zl_2} \ar[r]^(.62){1} & 
\overline{\zl_2} \ar@/_/[r]_(.5){1} & \overline{0} \ar@/_/[l]_1}$\\
circle $\zl_2+\zl_2$: $\xymatrix{\za \ar[r]^2 & \overline{0} 
\ar[rr]^(.3){1} & & *[l]{\hspace*{.5em}\overline{\zl_1+\zl_2}\text{
}}\ar@(dr,ur)[]_1 & \zb \ar@/_/[r]_2 & \overline{\zl_2}
\ar@/_/[l]_1}$}
\medskip

$n$ odd:\qquad
\parbox{.7\linewidth}{
circle 0: $\xymatrix{
\za \ar[r]^(.35){2} & \overline{\zl_1+\zl_2} \ar[r]^(.63){1} & \overline{0} 
  \ar@(dr,ur)[]_1 & & \zb \ar@/_/[r]_2 & \overline{\zl_2}\ar@/_/[l]_1}$\\
circle $\zl_1$: $\xymatrix{\za \ar[r]^(.45)2 & \overline{\zl_2} 
\ar[r]^(.5){1} &  \overline{0}\ar[r]^(.5){1} & 
\zb \ar@/_/[r]_2 & *[r]{\;\overline{\zl_1+\zl_2}} \ar@/_/[l]_1}$\\
circle $\zl_2$: $\xymatrix{\za \ar[r]^(.45)2 & \zb
\ar[r]^(.5){2} &  \overline{0}\ar[r]^(.5){1} & 
\overline{\zl_2} \ar@/_/[r]_1 & *[r]{\;\overline{\zl_1+\zl_2}} \ar@/_/[l]_1}$\\
circle $\zl_2+\zl_2$ (not NET): $\xymatrix{\za \ar[r]^2 & \overline{0} 
\ar[rr]^(.3){1} & & *[l]{\hspace*{.5em}\overline{\zl_1+\zl_2}\text{
}}\ar@(dr,ur)[]_1 & \zb \ar@(dr,ur)[]_2}$}
\medskip

We find that the four portraits for even integers $n$ are isomorphic
to the four portraits for odd integers $n$.  One of these four
portraits has only three postcritical points, and so the corresponding
Thurston map is not a NET map.  The other three Thurston maps are NET
maps whose dynamic portraits are mutually inequivalent, so these three
NET maps are mutually inequivalent.

\textbf{Modular group liftables associated to $\bD_0$.} 
For convenience, let us here set
  \begin{equation*}
G= \text{PSL}(2,\bbZ)\ltimes (\bbZ/2\bbZ)^2,
  \end{equation*}
viewed as the modular group of the square pillowcase.
Let $G_0$ denote the subgroup of liftables for $\bD_0$ in $G$, defined as the subgroup of elements $g \in G$ which lift under some (equivalently, any) NET map $F_0$ defined by the virtual diagram $\bD_0$ to another element $\widetilde{g} \in G$.  Generally, $g \mapsto \widetilde{g}$ is multivalued, due to the possible presence of deck transformations. But using
Proposition 3.5 of \cite{fpp2}, one verifies that the modular group
virtual multi-endomorphism associated to $\bD_0$ is actually an
endomorphism, that is, it is single valued.  Given $\zv\in G_0$, let
$\widetilde{\zv}$ denote its image under this modular group virtual
endomorphism. 

We define two affine maps $\zJ_0,\zJ_\infty \co \bbR^2\to \bbR^2$ so
that
  \begin{equation*}
\zJ_0(x)=
\left[\begin{matrix}1 & 2 \\ 0 & 1 \end{matrix}\right]x+
\left[\begin{matrix}2\\ 0\end{matrix}\right] \quad\text{ and }\quad
\zJ_\infty (x)=
\left[\begin{matrix}1 & 0 \\ -2 & 1 \end{matrix}\right]x.
  \end{equation*}
As in Proposition 3.2 of \cite{fpp2}, one verifies that
$\zJ_0$ and $\zJ_\infty $ determine two elements of $G_0$.  The images
of $\zJ_0$ and $\zJ_\infty $ in $\text{PSL}(2,\bbZ)$ generate
$\overline{\zG}(2)$, the subgroup arising from matrices congruent to the
identity modulo 2.

Now we consider the action of $G$ on slopes.  As in Section 6 of
\cite{fkklpps}, if $\zv\in G$ is represented by
$(\left[\begin{smallmatrix}a & b \\ c & d
\end{smallmatrix}\right],v)\in \text{SL}(2,\bbZ)\ltimes
(\bbZ/2\bbZ)^2$, then the pullback map $\zm_\zv$ induced by $\zv$ on
slopes of simple closed curves is represented by
$\left[\begin{smallmatrix}a & -c \\ -b & d
\end{smallmatrix}\right]$.  This and the previous two paragraphs imply
that the set of maps of the form $\zm_{\zv}$ for $\zv\in G_0$ contains
$\overline{\zG}(2)$.  

Continuing with slopes, let $\zm_0\co \overline{\bbQ}\to
\overline{\bbQ}\cup \{\odot \}$ be the slope function associated to
$\bD_0$.  Then, as usual, $\zm_0\circ \zm_\zv=\zm_{\widetilde{\zv}}
\circ \zm_0$ for every $\zv\in G_0$.  This leads us to define a group
homomorphism from $\overline{\zG}(2)$ to $\text{PSL}(2,\bbZ)$ so that
if $\zv\in G_0$, then this group homomorphism maps $\zh=\zm_\zv$ to
$\widehat{\zh}=\zm_{\widetilde{\zv}}$.  It follows that $\zm_0\circ
\zh=\widehat{\zh}\circ \zm_0$ for every $\zh\in \overline{\zG}(2)$.

We next define five elements $\zv_1$, $\zv_{-1}$, $\zv_0$, $\zv_\infty
$, $\zi$ in the $\text{PSL}(2,\bbZ)$-factor of $G$ by way of matrices
as follows:
  \begin{equation*}
\zv_1\leftrightarrow\left[\begin{matrix}0 & 1 \\ -1 & 2
\end{matrix}\right],\;
\zv_{-1}\leftrightarrow\left[\begin{matrix}2 & 1 \\ -1 & 0
\end{matrix}\right],\;
\zv_0\leftrightarrow\left[\begin{matrix}1 & 1 \\ 0 & 1
\end{matrix}\right],\;
\zv_\infty \leftrightarrow\left[\begin{matrix}1 & 0 \\ -1 & 1
\end{matrix}\right],\;
\zi\leftrightarrow\left[\begin{matrix}0 & -1 \\ 1 & 0
\end{matrix}\right].
  \end{equation*}
The first four are parabolic.  Their index is the slope which they fix.
Two elements in $\overline{\zG}(2)$ are
  \begin{equation*}
\zh_1=\zm_{\zv_1^2}\leftrightarrow\left[\begin{matrix}-1 & 2 \\ -2 & 3
\end{matrix}\right]\quad\text{ and }\quad \zh_0=\zm_{\zv_0 ^2} 
\leftrightarrow\left[\begin{matrix}1 & 0 \\ -2 & 1
\end{matrix}\right].
  \end{equation*}
We wish to compute $\widehat{\zh}_1$ and $\widehat{\zh}_0 $.  This
can be done as in Section 6 of \cite{fkklpps} by using functional
equations to compute two boundary values at the rationals, since this
suffices to determine an element of $\text{PSL}(2,\bbZ)$.  This computation
requires knowledge of some values of $\zm_0$.  We compute:
  \begin{equation*}
\zm_0\left(\frac{1}{0}\right)=\frac{0}{1},\quad
\zm_0\left(\frac{1}{1}\right)=\frac{-1}{1},\quad
\zm_0\left(\frac{-1}{1}\right)=\frac{1}{1},\quad
\zm_0\left(\frac{1}{2}\right)=\frac{1}{0}.
  \end{equation*}

This leads to the following two commutative diagrams of maps of pairs.
\begin{equation*}
\begin{CD}
\frac{1}{0}\;\frac{1}{1}@>\zm_0>>\frac{0}{1}\;\frac{-1}{1}\\
@V \zh_1 VV @VV \widehat{\zh}_1 V\\
\frac{1}{2}\;\frac{1}{1}@>\zm_0>>\frac{1}{0}\;\frac{-1}{1}
\end{CD}\qquad\qquad
\begin{CD}
\frac{1}{1}\;\frac{1}{2}@>\zm_0>>\frac{-1}{1}\;\frac{1}{0}\\
@V \zh_0 VV @VV \widehat{\zh}_0 V\\
\frac{-1}{1}\;\frac{1}{0}@>\zm_0>>\frac{1}{1}\;\frac{0}{-1}
\end{CD}
  \end{equation*}
We conclude that
  \begin{equation*}
\widehat{\zh}_1\leftrightarrow
\left[\begin{matrix}1 & -1 \\ 0 & 1 \end{matrix}\right]
\left[\begin{matrix}0 & -1 \\ 1 & 1 \end{matrix}\right]^{-1}=
\left[\begin{matrix}1 & -1 \\ 0 & 1 \end{matrix}\right]
\left[\begin{matrix}1 & 1 \\ -1 & 0 \end{matrix}\right]=
\left[\begin{matrix}2 & 1 \\ -1 & 0 \end{matrix}\right]
\leftrightarrow\zm_{\zv_{-1}}
  \end{equation*}
and
  \begin{equation*}
\widehat{\zh}_0\leftrightarrow
\left[\begin{matrix}1 & 0 \\ 1 & -1 \end{matrix}\right]
\left[\begin{matrix}-1 & 1 \\ 1 & 0 \end{matrix}\right]^{-1}=
\left[\begin{matrix}1 & 0 \\ 1 & -1 \end{matrix}\right]
\left[\begin{matrix}0 & 1 \\ 1 & 1 \end{matrix}\right]=
\left[\begin{matrix}0 & 1 \\ -1 & 0 \end{matrix}\right]
\leftrightarrow\zm_\zi.
  \end{equation*}

\textbf{Translation equivalence.} We say that NET maps $F$ and $F'$ on
the square pillowcase are translation equivalent if and only if there
exists a translation $\zt=(1,v)\in G$  such that $\zt\circ F$ and $F'$
are isotopic.  This is an equivalence relation because the set of
translations is a subgroup of the modular group.  We write $F\approx
F'$ to signify that $F$ and $F'$ are translation equivalent.  Changing
the circled lattice point in a NET map presentation diagram changes
the original NET map to one which is translation equivalent to it.

Because the set of translations is a normal subgroup of the modular
group, both twisting and conjugating by an element of the modular
group respects translation equivalence.  To be specific, let $F$ and
$F'$ be NET maps on the square pillowcase, and let $\zj\in G$.  If
$F'=\zj^{-1}\circ F\circ \zj$, then conjugation by $\zj$ maps the set
of NET maps translation equivalent to $F$ to the set of NET maps
translation equivalent to $F'$.  Similarly, if $F'=\zj\circ F$, then
postcomposition by $\zj$ maps the set of NET maps translation
equivalent to $F$ to the set of NET maps translation equivalent to
$F'$.

Let $F_0$ be a NET map arising from $\bD_0$ (by circling either $\zl_1$,
$\zl_2$ or $\zl_1+\zl_2$).  The last two displays imply that
  \begin{equation*}
\zv_1^2\circ F_0\approx F_0\circ \zv_{-1}\quad\text{ and }\quad
\zv_0^2\circ F_0\approx F_0\circ \zi.
  \end{equation*}

\textbf{Few equivalence classes.} The following theorem implies that
the presentation diagrams $\bD_n$ produce very few equivalence classes
of NET maps.  Theorem~\ref{thm:conjugation2} makes this precise for
the case in which $n\equiv 2\text{ mod }3$.  To simplify notation, we
use juxtaposition of functions to signify composition.

\begin{thm}\label{thm:conjugation1} For an integer $m \geq 0$ let $F_m$ be a NET map arising from $\bD_m$ and let $\zj_m=\zv_1^m \zi$.
Then $\zj_m^{-1}F_m \zj_m \approx
F_{2m+1}$ and $\zj_m^{-1}F_{m+1}\zj_m \approx F_{2m}$.
\end{thm}
  \begin{proof} The twist relationship between $\bD_0$ and $\bD_m$
implies that $F_m\approx \zv_1^mF_0$.  One verifies the following
identities.
  \begin{equation*}
\zi^2=1\quad \zi \zv_1\zi=\zv_{-1}\quad \zi \zv_0^{-2}=\zv_1
\quad \zi \zv_0^2=\zv_{-1}^{-1}
  \end{equation*}
Using this, the display immediately preceding this theorem and the
fact that conjugating and twisting respect translation equivalence,
the first assertion can be proved as follows.
  \begin{align*}
\zj_m^{-1}F_m \zj_m & \approx (\zv_1^m \zi )^{-1}(\zv_1^mF_0)(\zv_1^m \zi)=\zi
F_0\zi(\zi \zv _1^m \zi)=\zi F_0\zi^{-1} \zv_{-1}^m\\
 & \approx \zi
\zv_0^{-2}\zv_1^{2m}F_0=\zv_1\zv_1^{2m}F_0=\zv_1^{2m+1}F_0\approx F_{2m+1}
  \end{align*}
The second assertion can be proved similarly after observing that
$\zj_m=\zv_1^{m+1}\zv_0^2$.
  \begin{align*}
\zj_m^{-1}F_{m+1}\zj_m & \approx
(\zv_1^{m+1}\zv_0^2)^{-1}(\zv_1^{m+1}F_0)(\zv_1^{m+1}\zv_0^2)=
\zv_0^{-2}F_0\zi(\zi \zv_1^{m+1}\zi)(\zi \zv_0^2)\\
 & \approx \zv_0^{-2}F_0\zi
\zv_{-1}^{m+1}\zv_{-1}^{-1}\approx \zv_0^{-2}\zv_0^2
\zv_1^{2m}F_0=\zv_1^{2m}F_0\approx F_{2m}
  \end{align*}

\end{proof}

We are going to analyze the sequence of sets of NET maps arising from diagrams $\mathbf{D}_n$ with $n \equiv 2\text{ mod }3$. In order to apply the previous Theorem \ref{thm:conjugation1}, there will be two cases depending on whether $n$ is even or odd.

\begin{thm}\label{thm:conjugation2} Let $n$ be a positive integer such
that $n\equiv 2\text{ mod }3$.  Then every NET map arising from
$\bD_n$ is conjugate to a NET map arising from $\bD_2$.
\end{thm}
  \begin{proof} Suppose $n\equiv 2\text{ mod }3$ and $F_n$ is a NET map arising from $\mathbf{D}_n$.  Suppose first $n$ is odd. Write $n=2m+1$ and note that also $m\equiv 2\text{ mod }3$. Theorem \ref{thm:conjugation1} implies $F_n$ is equivalent to a map $F_m$ produced from $\mathbf{D}_m=\mathbf{D}_{\frac{n-1}{2}}$.  If $n=2m$ is even then $m+1\equiv 2\text{ mod }3$ and similarly $F_n$ is equivalent to a map  $F_{m+1}$ produced from $\mathbf{D}_{m+1}=\mathbf{D}_{\frac{n}{2}+1}$. Induction yields the result. 
%  
%  
%  Note that if $m$ is a positive integer such that
%$n=2m+1$, then $m\equiv 2\text{ mod }3$.  Similarly, if $m$ is a
%positive integer such that $n=2m$, then $m+1\equiv 2\text{ mod }3$.
%Theorem~\ref{thm:conjugation2} follows from this, the fact that
%conjugation respects translation equivalence and a straightforward
%induction argument using Theorem~\ref{thm:conjugation1}.
\end{proof}

\textbf{The obstruction theorem.} Here is the main result of this
section.

\begin{thm}\label{thm:obstrn} Let $n$ be a nonnegative integer with
$n\equiv 2\text{ mod } 3$.  Then every NET map arising from $\bD_n$ is
obstructed.  Let $s_n$ be the slope of this obstruction.  For a nonnegative integer $m$ let $\zj_m=\zv_1^m \zi$ and 
$\zn_m=\zm_{\zj_m}$.
Then:
\begin{enumerate}
  \item $s_2=\frac{1}{0}$
  \item If $n=2m+1$ is odd, then
$s_n=\zn_m(s_m)$.
  \item If $n=2m$ is even, then
$s_n=\zn_m(s_{m+1})$.
\end{enumerate}
\end{thm}
  \begin{proof} Using Figure~\ref{fig:prendgm}, it is easy to see for
every NET map arising from $\bD_2$ that every simple closed curve
$\zg$ with slope $\frac{1}{0}$ pulls back to a simple closed curve
with slope $\frac{1}{0}$.  Only one connected component of the
preimage is essential and nonperipheral, and it maps to $\zg$ with
degree 1.  So $\zg$ is an obstruction with multiplier 1.  This proves
that every NET map arising from $\bD_2$ is obstructed with obstruction
slope $s_2=\frac{1}{0}$.  

Now suppose $n\equiv 2\text{ mod } 3$.  Statements 2 and 3 follow from Theorem~\ref{thm:conjugation1}: conjugation by ${\zj_m}$ sends the obstruction of $F_m$ to the obstruction of $F_n$, and the presence of and location of obstructions is unaffected by composition with translations.  Finally, a straightforward induction argument proves that every NET
map arising from $\bD_n$ is obstructed.  This proves
Theorem~\ref{thm:obstrn}.

\end{proof}

Recall that if $\zv\in G$ is represented by
$(\left[\begin{smallmatrix}a & b \\ c & d
\end{smallmatrix}\right],t)\in \text{SL}(2,\bbZ)\ltimes
(\bbZ/2\bbZ)^2$, then $\zm_\zv$ is represented by
$\left[\begin{smallmatrix}a & -c \\ -b & d
\end{smallmatrix}\right]$.  Using this and
line~\ref{line:twistmatrix}, we find that
  \begin{equation*}
\zn_m=\zm_{\zv_1^m \zi}=\zm_\zi \zm_{\zv_1^m}\leftrightarrow
\left[\begin{matrix}0 & -1 \\ 1 & 0 \end{matrix}\right]
\left[\begin{matrix}-m+1 & m \\ -m & m+1 \end{matrix}\right]=
\left[\begin{matrix}m & -m-1 \\ -m+1 & m \end{matrix}\right].
\end{equation*}
This explicit form of $\zn_m$ and Theorem~\ref{thm:obstrn} allow for
the computation of $s_n$ for small values of $n$.  See
Table~\ref{table:small}.

\renewcommand{\arraystretch}{1.5}
\begin{table}
\begin{center}\begin{tabular}{c|cccccccccc}
$n$ & 2 & 5 & 8 & 11 & 14 & 17 & 20 & 23 & 26 & 29\\ \hline 
$s_n$ & $\frac{1}{0}$ & $\frac{-2}{1}$ & $\frac{-13}{10}$
 & $\frac{-16}{13}$ & $\frac{-171}{148}$ & $\frac{-194}{171}$
 & $\frac{-303}{274}$ & $\frac{-332}{303}$ & $\frac{-4295}{3976}$
 & $\frac{-4614}{4295}$\\
\end{tabular}\end{center}
\medskip
\caption{Some values of $s_n$}\label{table:small}
\end{table}

\textbf{Asymptotics of $s_n$.} We begin to estimate $s_n$ with the
following lemma.

\begin{lemma}\label{lemma:est} Let $n$ be an integer with $n\equiv
2\text{ mod }3$ and $n\ge 5$.  Then the following statements hold.
\begin{enumerate}
  \item $-2\le s_n<-1$
  \item If $n=2m+1$ for an integer $m$, then $m\height(s_m)\le
\height(s_n)\le 2(m+1)\height(s_m)$.
  \item If $n=2m$ for an integer $m$, then $m\height(s_{m+1})\le
\height(s_n)\le 2(m+1)\height(s_{m+1})$.
\end{enumerate}
\end{lemma}
  \begin{proof} Let $m$ be an integer with $m\ge 2$.  Using
the above explicit form of $\zn_m$, we find that
  \begin{equation*}
-2\le \frac{m}{-m+1}=\zn_m(\infty )\quad\text{ and }\quad
\zn_m(-1)=\frac{-2m-1}{2m-1}<-1.
  \end{equation*}
Since $\zn_m$ preserves orientation, it follows that $\zn_m$ maps the
interval $(-\infty,-1]$ into the interval $[-2,-1)$.  This,
Theorem~\ref{thm:obstrn} and induction prove statement 1.  Statement 1
and Theorem~\ref{thm:obstrn} easily imply statements 2 and 3.  This
proves Lemma~\ref{lemma:est}.

\end{proof}

Statements 2 and 3 of Lemma~\ref{lemma:est} lead to two more lemmas.
Here is the first of these.

\begin{lemma}\label{lemma:estf} Let $f\co (0,\infty )\to (0,\infty )$
be a function which is constant on the interval $(0,4)$ such that
$f(x)=K(\frac{x}{2}-1)f(\frac{x}{2}-1)$ for some positive real number
$K$ and $x\ge 4$.  Then there exists a positive real number $c$ such
that $f(x)>x^{\log_2(x)/2-c}$ for every sufficiently large real number $x$.
\end{lemma}
  \begin{proof} The assumptions imply that there exists a positive
real number $L$ such that if $x\ge 4$, then
  \begin{equation*}
f(x)=L[K(\frac{x}{2}-1)][K(\frac{x}{4}-\frac{1}{2}-1)]
[K(\frac{x}{8}-\frac{1}{4}-\frac{1}{2}-1)]\cdots ,
  \end{equation*}
where the number $r$ of pairs of brackets is either $\left\lfloor
\log_2(x)\right\rfloor -1$ or $\left\lfloor \log_2(x)\right\rfloor
-2$.  The last bracketed term is at least $K$, and every bracketed
term is more than 2 times the succeeding term.  So
  \begin{equation*}
f(x) >LK^r2^{1+\cdots +(r-1)}=LK^r2^{r(r-1)/2}.
  \end{equation*}
Since $2^{r+3}>x$, it follows that $2^r>\frac{x}{8}$.  Hence
  \begin{equation*}
f(x)>LK^r8^{-(r-1)/2}x^{(r-1)/2}>x^{\log_2(x)/2-c}
  \end{equation*}
for some positive real number $c$ and every sufficiently large real
number $x$.

This proves Lemma~\ref{lemma:estf}.

\end{proof}

\begin{lemma}\label{lemma:estg} Let $g\co (0,\infty )\to (0,\infty )$
be a function which is constant on the interval $(0,4)$ such that
$g(x)=K(\frac{x}{2}+1)g(\frac{x}{2}+1)$ for some positive real number $K$
and $x\ge 4$.  Then there exists a positive real number $c$ such that
$g(x)<x^{\log_2(x)/2+c}$ for every sufficiently large real number $x$.
\end{lemma}
  \begin{proof} The assumptions imply that there exists a positive
real number $L$ such that if $x\ge 4$, then
  \begin{equation*}
g(x)=L[K(\frac{x}{2}+1)][K(\frac{x}{4}+\frac{1}{2}+1)]
[K(\frac{x}{8}+\frac{1}{4}+\frac{1}{2}+1)]\cdots ,
  \end{equation*}
where the number of pairs of brackets is $r=\left\lfloor
\log_2(x)\right\rfloor -1$.  Each of the terms of the form
$\frac{x}{2^k}$ is at least 2, which is greater than the 
finite geometric series which follows it.  So
  \begin{equation*}
g(x)<L(2K)^r\frac{x}{2}\frac{x}{4}\frac{x}{8}\cdots
=L(2Kx)^r2^{-1-2-3-\cdots -r} =L(2Kx)^r2^{-r(r+1)/2}.
  \end{equation*}
Since $2^{r+2}>x$, we have that $2^r>\frac{x}{4}$.  Hence
  \begin{equation*}
g(x)<L(2Kx)^r\left(\frac{4}{x}\right)^{(r+1)/2}
= 2^{2r+1}K^rLx^{(r-1)/2}<x^{\log_2(x)/2+c}
  \end{equation*}
for some positive real number $c$ and every sufficiently large real
number $x$.

This proves Lemma~\ref{lemma:estg}.

\end{proof}

Now we can estimate the height of $s_n$.

\begin{thm}\label{thm:est} There exists a positive real number $c$
such that
  \begin{equation*}
n^{\log_2(n)/2-c}<\height(s_n)<n^{\log_2(n)/2+c}
  \end{equation*}
for every sufficiently large positive integer $n$ such that $n\equiv
2\text{ mod }3$.
\end{thm}
  \begin{proof} To obtain the lower bound on $\height(s_n)$, we prove
that there exists a function $f$ as in Lemma~\ref{lemma:estf} such
that $\height(s_n)\ge f(n)$ for every nonnegative integer $n$ with
$n\equiv 2\text{ mod } 3$.  
Define $f\co (0,\infty) \mapsto (0,\infty)$ recursively by
$f(x) = 1$ if $x\in (0,4)$ and $f(x) = (\frac{x}{2}-1) f(\frac{x}{2}-1)$ if
$x\ge 4$.
We have that $
\height(s_2)\ge f(2)$ because $\height(s_2)=1=f(2)$.  Now suppose that
$n=2m+1$, that $m\equiv 2\text{ mod }3$, that $m\ge 2$ and that
$\height(s_m)\ge f(m)$.  Using Lemma~\ref{lemma:est} and the fact that
$f$ is increasing,
  \begin{equation*}
\height(s_n)\ge m\height(s_m)\ge mf(m)=f(2m+2)\ge f(2m+1)=f(n),
  \end{equation*}
as desired.  If $n=2m$ and $\height(s_{m+1})\ge f(m+1)$, then
  \begin{equation*}
\height(s_n)\ge m\height(s_{m+1})\ge mf(m+1)\ge mf(m)=f(2m+2)\ge
f(2m)=f(n).
  \end{equation*}
This establishes the lower bound on $\height(s_n)$.  

For the upper bound, we apply Lemma~\ref{lemma:estg} with the function
$g$ for which $K=2$ and $g(x)=1$ for $x\in (0,4)$.  Because the
argument for the upper bound is so similar to the argument for the
lower bound, we simply give the estimates which establish the upper
bound:
  \begin{align*}
\height(s_n)&\le 2(m+1)\height(s_m)\le 2(m+1)g(m)\\
 &\le 2(m+1)g(m+1)= g(2m)\le g(2m+1)=g(n)
  \end{align*}
and
  \begin{equation*}
\height(s_n)\le 2(m+1)\height(s_{m+1})\le 2(m+1)g(m+1)
 = g(2m)=g(n).
  \end{equation*}

This completes the proof of Theorem~\ref{thm:est}.

\end{proof}

Theorem~\ref{thm:est} shows that the function $H(C)$ in
Theorem~\ref{thm:finite} cannot be taken to be a polynomial in $C$
even if we restrict attention to only negative reciprocals of slopes
of obstructions.  This completes the proof of
Theorem~\ref{thm:finite}.

\end{document}